\documentclass[11pt,reqno,oneside]{article}
\usepackage{amsmath, amsthm, mathtools, amssymb, enumerate,esint}
\usepackage{graphicx}
\usepackage{xcolor}
\usepackage{caption}
\usepackage{datetime}
\usepackage{fancyhdr}
\usepackage{marginnote}

\usepackage[unicode,breaklinks=true,colorlinks=true,linkcolor=blue,urlcolor=blue,citecolor=blue]{hyperref}
\usepackage[margin=1in,headheight=14pt]{geometry}
\usepackage{mdframed}
\usepackage{cancel}
\numberwithin{equation}{section}

  \chardef\forshowkeys=1
  \chardef\showllabel=0
  \chardef\refcheck=0
  \chardef\sketches=0
  \chardef\showcolors=1

\ifnum\showllabel=1
 \def\llabel#1{\marginnote{\color{lightgray}\rm\small(#1)}[-0.0cm]\notag}
\else
 \def\llabel#1{\notag}
\fi

\begin{document}
\def\tdot{{\gdot}}
\def\intk{[k T_0, (k+1)T_0]}
\def\Dg{{D'g}}
\def\ua{u^{\alpha}}
\def\ques{{\colr \underline{??????}\colb}}
\def\nto#1{{\colC \footnote{\em \colC #1}}}
\def\fractext#1#2{{#1}/{#2}}
\def\fracsm#1#2{{\textstyle{\frac{#1}{#2}}}}   
\def\baru{U}
\def\nnonumber{}
\def\palpha{p_{\alpha}}
\def\valpha{v_{\alpha}}
\def\qalpha{q_{\alpha}}
\def\walpha{w_{\alpha}}
\def\falpha{f_{\alpha}}
\def\dalpha{d_{\alpha}}
\def\galpha{g_{\alpha}}
\def\halpha{h_{\alpha}}
\def\psialpha{\psi_{\alpha}}
\def\psibeta{\psi_{\beta}}
\def\betaalpha{\beta_{\alpha}}
\def\gammaalpha{\gamma_{\alpha}}
\def\TTalpha{T_{\alpha}}
\def\TTalphak{T_{\alpha,k}}
\def\falphak{f^{k}_{\alpha}}
\def\R{\mathbb R}
\newcommand {\Dn}[1]{\frac{\partial #1  }{\partial N}}
\def\andand{\text{\quad and\quad}}
\def\wherewhere{\text{\quad where\quad}}
\def\withwith{\text{\quad with\quad}}
\def\andandone{\text{\, and\quad}}
\def\mm{m}
\def\colr{{}}
\def\colr{\color{red}}
\def\colu{\color{blue}}
\def\bnew{\color{red}}
\def\enew{\color{black}}
\def\bold{\color{blue}}
\def\eold{\color{black}}
\def\colg{\color{green}}
\def\colb{{}}
\def\colb{\color{black}}
\def\cole{\color{blue}}
\def\cole{\color{black}}
\def\colA{{}}
\def\colB{{}}
\def\colC{{}}
\def\colD{{}}
\def\colE{{}}
\def\colF{{}}

\def\rref#1{{\ref{#1}{\rm \tiny \fbox{\tiny #1}}}}
\def\theequation{\fbox{\bf \thesection.\arabic{equation}}}
\def\ccite#1{{\cite{#1}{\rm \tiny ({#1})}}}
\def\startnewsection#1#2{\newpage\colg \section{#1}\colb\label{#2}}
\setcounter{equation}{0}
\pagestyle{fancy}
\cfoot{}
\rfoot{\thepage}
\chead{}
\rhead{\thepage}
\def\nnewpage{\newpage}
\newcounter{startcurrpage}
\newcounter{currpage}
\def\llll#1{{\rm\tiny\fbox{#1}}}
   \def\blackdot{{\color{red}{\hskip-.0truecm\rule[-1mm]{4mm}{4mm}\hskip.2truecm}}\hskip-.3truecm}
   \def\bdot{{\color{blue} {\hskip-.0truecm\rule[-1mm]{4mm}{4mm}\hskip.2truecm}}\hskip-.3truecm}
   \def\purpledot{{\colA{\rule[0mm]{4mm}{4mm}}\colb}}
   \def\pdot{\purpledot}
   \def\gdot{{\color{green}{\rule[0mm]{4mm}{4mm}}\colb}}

\def\nts#1{{\hbox{\bf ~#1~}}} 
\def\igor#1{{\colu\bf{\hbox{\bf IK: ~#1~}}}} 
\def\qi#1{{\hbox{\bf\color{purple} QX: ~#1~}}} 
\def\wojtek#1{{\hbox{\bf WO: ~#1~}}} 
\def\nts#1{{\colr\small\hbox{\bf ~#1~}}} 
\def\ntsf#1{\footnote{\colb\hbox{\rm ~#1~}}} 
\def\bigline#1{~\\\hskip2truecm~~~~{#1}{#1}{#1}{#1}{#1}{#1}{#1}{#1}{#1}{#1}{#1}{#1}{#1}{#1}{#1}{#1}{#1}{#1}{#1}{#1}{#1}\\}
\def\biglineb{\bigline{$\downarrow\,$ $\downarrow\,$}}
\def\biglinem{\bigline{---}}
\def\biglinee{\bigline{$\uparrow\,$ $\uparrow\,$}}
\def\inon#1{\hbox{\ \ \ \ \ }\hbox{#1}}
\def\onon#1{\inon{on~$#1$}}
\def\inin#1{\inon{in~$#1$}}
\def\Omf{\Omega_{\text f}}
\def\Ome{\Omega_{\text e}}
\def\Gae{\Gamma_{\text e}}
\def\Gaf{\Gamma_{\text f}}
\def\Gac{\Gamma_{\text c}}
\def\wext{\tilde{w}}
\def\wexts{\widetilde{S w}}
\def\wexta{\overline{w}}
\def\wextb{\overline{\overline{w}}}
\def\mbar{{\overline M}}
\def\tilde{\widetilde}

\newtheorem{Theorem}{Theorem}[section]
\newtheorem{Corollary}[Theorem]{Corollary}
\newtheorem{Proposition}[Theorem]{Proposition}
\newtheorem{Lemma}[Theorem]{Lemma}
\newtheorem{Remark}[Theorem]{Remark}
\newtheorem{definition}{Definition}[section]

\def\theequation{\thesection.\arabic{equation}}
\def\endproof{\hfill$\Box$\\}
\def\square{\hfill$\Box$\\}
\def\comma{ {\rm ,\qquad{}} }            
\def\commaone{ {\rm ,\quad{}} }         
\def\dist{\mathop{\rm dist}\nolimits}    
\def\sgn{\mathop{\rm sgn\,}\nolimits}    
\def\Tr{\mathop{\rm Tr}\nolimits}    
\def\curl{\mathop{\rm curl}\nolimits}    
\def\div{\mathop{\rm div}\nolimits}    
\def\supp{\mathop{\rm supp}\nolimits}    
\def\divtwo{\mathop{{\rm div}_2\,}\nolimits}    
\def\re{\mathop{\rm {\mathbb R}e}\nolimits}    
\def\indeq{\qquad{}\!\!\!\!}                     
\def\period{.}                           
\def\semicolon{\,;}                      
\newcommand{\cD}{\mathcal{D}}
\newcommand{\eqnb}{\begin{equation}}
\newcommand{\eqne}{\end{equation}}
\newcommand{\na}{\nabla }
\newcommand{\bog}{b}
\newcommand{\bb}{\nu }
\newcommand{\ww}{{\overline{w}}}
\newcommand{\la}{\lambda }
\newcommand{\p}{\partial }
\renewcommand{\d}{\mathrm{d} }
\newcommand{\N}{\mathbb{N}}
\newcommand{\on}{\omega^{(n)}}
\newcommand{\onn}{\omega^{(n+1)}}
\newcommand{\onm}{\omega^{(n-1)}}
\newcommand{\nn}{\mathsf{n}}
\newcommand{\ee}{\mathrm{e}}
\newcommand{\T}{\mathbb{T}}
\renewcommand{\R}{\mathbb{R}}
\newcommand{\lec}{\lesssim  }
\newcommand{\gec}{\gtrsim  }
\newcommand{\tom}{\tilde{\omega}}
\newcommand{\tu}{\tilde{u}}
\newcommand{\un}{^{(n)}}
\newcommand{\unp}{^{(n+1)}}
\newcommand{\unm}{^{(n-1)}}
\newcommand{\lo}{L^2(\Omega)}
\newcommand{\lio}{L^\infty(\Omega)}
\newcommand{\ls}{L^2(S)}
\newcommand{\lis}{L^\infty(S)}
\newcommand{\crit}{\mathrm{crit}}
\newcommand{\bdd}{\mathrm{bd}}
\newcommand{\rin}{\mathrm{in}}
\newcommand{\out}{\mathrm{out}}
\def\NL{\text{NL}}

\title{Quantitative Decay for Linear Parabolic Equations}
\author{Igor~Kukavica and Qi Xu}

\date{}
\maketitle
\medskip
\begin{abstract}
We establish sharp temporal lower bounds for the full spatial
$L^2$-norm of energy solutions to $\partial_tu-\Delta u=vu$ on $\R^n$,
with $n\geq3$. For global solutions, nonvanishing at a single time implies
the lower bound $ce^{-Ct}$ when $v$ is small in the scale-invariant
space $L^\infty_tL^{n/2}_x$, and $ce^{-Ct^2}$ throughout the
subcritical range $v\in L^\infty_tL^p_x$, with $p>n/2$, without
smallness. The linear time exponent is optimal already for the free
heat equation, while we show that the quadratic exponent is optimal
for bounded complex
potentials on~$\R^3$. No additional regularity of the potential is
required. The proof uses a convexified mixed Carleman estimate whose
norms are designed together with the final time truncation.
Integration over spatial centers removes the heat-kernel weight and
yields quantitative comparisons of the full solution norms.
To prove quadratic sharpness in the Euclidean energy class, we
construct a time-dependent function $g$ that localizes Meshkov's
periodic parabolic example while keeping the induced potential
bounded across the zeros of the periodic profile.
\end{abstract}
\colb

\noindent\thanks{\em Keywords:\/}
Parabolic equations, quantitative decay,
critical potentials, Carleman estimates, backward uniqueness.

\section{Introduction}\label{sec01}
We study the long-time decay of nonzero solutions of
\begin{align*}
  \partial_tu-\Delta u=v(t,x)u
  \qquad\text{on }(T_0,\infty)\times\R^n,
  \qquad n\geq3,
\end{align*}
when the potential belongs to $L^\infty_tL^p_x$, with $p\geq n/2$.
The equation is understood in the sense of distributions, and we work in the
energy class
\begin{align*}
  u\in C([T_0,T_*];L^2(\R^n))
  \cap L^2((T_0,T_*);H^1(\R^n)),
\end{align*}
for every finite $T_*>T_0$. Both $u$ and $v$ may be complex-valued.
The question is how much
spatial integrability of the potential suffices to prevent rapid decay
of the full $L^2$-norm, without additional assumptions on the
potential's regularity or the solution's frequency and concentration.

Our main results have a simple consequence: If the $L^2$-norm is
nonzero at one time, then it admits an exponential or Gaussian lower
bound at all sufficiently late times. More precisely, suppose that
$\|u(T_0)\|_{L^2_x}>0$. For a sufficiently small critical potential,
\begin{align*}
  \|v\|_{L^\infty_tL^{n/2}_x}\leq\epsilon_n
  \quad\Longrightarrow\quad
  \|u(t)\|_{L^2_x}\geq ce^{-Ct},
  \qquad t\geq T_0+1.
\end{align*}
For every subcritical exponent $p>n/2$, including $p=\infty$,
no smallness is needed and we have
\begin{align*}
  v\in L^\infty_tL^p_x
  \quad\Longrightarrow\quad
  \|u(t)\|_{L^2_x}\geq ce^{-Ct^2},
  \qquad t\geq T_0+1.
\end{align*}
Here the potential norms are taken over $(T_0,\infty)\times\R^n$,
$\epsilon_n>0$ depends only on $n$, and $c,C>0$ may depend on the
solution and the potential but are independent of~$t$.
Theorems~\ref*{T01} and~\ref*{T02} give quantitative comparisons
on finite time intervals, with explicit dependence on two earlier
solution norms; the formulation above is Corollary~\ref*{C01}.

The linear time exponent is optimal already for the free heat
equation: Initial data with Fourier support in a fixed annulus
separated from the origin give exponential decay. The quadratic
exponent is optimal for bounded complex potentials on $\R^3$, as
shown by the global energy solution in Theorem~\ref*{T03}.
The subcritical estimate thus extends the bounded-potential Gaussian
scale to every $p>n/2$, whereas smallness at the critical
integrability yields the sharper exponential scale.

The connection between decay and uniqueness goes back to
Lax~\cite{Lax1956}, Lions and Malgrange~\cite{LionsMalgrange1960},
and the Hilbert-space method of Agmon and
Nirenberg~\cite{AgmonNirenberg1967}; see also
Ogawa~\cite{Ogawa1967}. The bounded-potential calculation explains
both the Gaussian time scale and the difficulty with integrable
potentials. For a sufficiently regular solution with positive norm,
set
\begin{align*}
  Q(t)=\frac{\|\nabla u(t)\|_{L^2_x}^2}{\|u(t)\|_{L^2_x}^2}
\end{align*}
for the Dirichlet quotient, introduced by Agmon and Nirenberg
in~\cite{AgmonNirenberg1967}.
With $M=\|v\|_{L^\infty_{t,x}}$, direct differentiation and
completion of the square give
\begin{align*}
\begin{split}
  Q'(t)
  &=-\frac{2\|\Delta u+Q(t)u+\frac12vu\|_{L^2_x}^2}
  {\|u\|_{L^2_x}^2}
  +\frac{\|vu\|_{L^2_x}^2}{2\|u\|_{L^2_x}^2}
  \leq\frac{M^2}{2}
\end{split}
\end{align*}
and
\begin{align*}
\begin{split}
  \frac{\text{d}}{\text{d}t}\log\|u(t)\|_{L^2_x}
  &=-Q(t)+\frac{\operatorname{Re}\int_{\R^n}v|u|^2\,\text{d}x}
  {\|u\|_{L^2_x}^2}
  \geq-Q(t)-M.
\end{split}
\end{align*}
All quantities on the right are evaluated at~$t$.
The quotient grows at most linearly, and integration gives a
quadratic exponent in the lower bound. Averaging the second identity
over an earlier interval replaces the initial quotient by two
solution norms. This calculation applies equally to complex-valued
solutions and potentials.

At the critical exponent, Sobolev embedding gives 
\begin{align*}
\begin{split}
  \left|\int_{\R^n}v|u|^2\,\text{d}x\right|
  &\leq C\|v\|_{L^{n/2}_x}\|\nabla u\|_{L^2_x}^2
\end{split}
\end{align*}
and
\begin{align*}
\begin{split}
  \|vu\|_{L^2_tL^{\frac{2n}{n+2}}_x}
  &\leq C\|v\|_{L^\infty_tL^{n/2}_x}
  \|\nabla u\|_{L^2_{t,x}}.
\end{split}
\end{align*}
The first estimate yields energy monotonicity under smallness, but
the second places the source in a dual Sobolev space, not in the
$L^2$ space required by the quotient identity. For finite subcritical~$p$, the ratio $\|vu\|_{L^2_x}/\|u\|_{L^2_x}$ is likewise not
controlled by $\|v\|_{L^p_x}$ alone. Retaining a linear exponent at
the small critical endpoint and a quadratic exponent for all
$p>n/2$ therefore requires an estimate compatible with this weaker
source control. Related backward uniqueness results for lower order
perturbations appear in~\cite{Kukavica2004}.
Del Santo and Prizzi~\cite{DelSantoPrizzi2007} exclude superexponential decay
under time integrability assumptions and suitable time regularity
of the principal coefficients.

The appropriate spatial estimates belong to the theory of unique
continuation. The foundational work of
Carleman~\cite{Carleman1939} and Aronszajn~\cite{Aronszajn1957},
and H\"ormander's general theory~\cite{Hormander1985}, established
the role of weighted inequalities. For rough potentials, the spatial
norms in these inequalities are essential. Jerison and
Kenig~\cite{JerisonKenig1985} proved strong unique continuation at
the critical elliptic integrability $L^{n/2}_{\mathrm{loc}}$,
while the uniform Sobolev estimates of Kenig, Ruiz, and
Sogge~\cite{KenigRuizSogge1987} provide a basic mechanism for
controlling rough lower order terms in different spatial norms.

For parabolic equations, Sogge~\cite{Sogge1990} and
Escauriaza~\cite{Escauriaza2000} developed Carleman estimates for
integrable potentials. Escauriaza and Vega~\cite{EscauriazaVega2001}
reached the critical spatial exponent $n/2$ through Hermite
resolvent estimates and a Littlewood-Paley decomposition in time.
Their results include strong unique continuation for small
$L^\infty_tL^{n/2}_x$ potentials, under the corresponding vanishing
and growth hypotheses. Koch and Tataru~\cite{KT2009} developed a
framework combining convexity, Hermite estimates, localized
parametrices, and duality for operators with nonsmooth principal
coefficients and unbounded lower order terms; their Hermite
bounds~\cite{KochTataru2005Hermite} are part of this theory.
Jeong, Lee, and Ryu~\cite{JeongLeeRyu2024} subsequently obtained
strong unique continuation for small potentials in the larger weak
space~$L^\infty_tL^{n/2,\infty}_x$.

A complementary approach uses frequency functions and monotonicity.
The elliptic work of Garofalo and Lin~\cite{GarofaloLin1986}
has parabolic counterparts in the work of Lin, Kurata, Poon, and
Chen~\cite{Lin1990,Kurata1994,Poon1996,Chen1998}, while Colding and
Minicozzi~\cite{ColdingMinicozzi2022} established parabolic frequency
monotonicity on manifolds. Variable-coefficient continuation and
backward uniqueness on unbounded domains were studied by
Escauriaza and Fern\'andez~\cite{EscauriazaFernandez2003} and
Escauriaza, Seregin, and \v{S}ver\'ak~\cite{EscauriazaSereginSverak2003}.
For quantitative bounds, Camliyurt and one of the
authors~\cite{CamliyurtKukavica2018} treated vanishing order with
bounded lower order coefficients, and~\cite{KukavicaLe2022} obtained vanishing-order and fixed-time doubling
estimates for integrable coefficients, including potentials with
$p>2n/3$.

Quantitative continuation also connects to nodal set bounds, as in
Donnelly and Fefferman~\cite{DonnellyFefferman1988}, and to spectral
localization, as in Bourgain and Kenig~\cite{BourgainKenig2005}.
For parabolic equations, Escauriaza, Kenig, Ponce, and Vega proved
lower bounds at spatial infinity~\cite{EKPV2006} and sharp Gaussian
uncertainty principles using logarithmic convexity~\cite{EKPV2016}.
The latter allow bounded complex potentials in the energy class.
These spatial decay questions and the local vanishing results above
provide the background for our temporal problem: quantitative
nonvanishing of the full norm at every later time, with no prescribed
spatial localization of the solution.

Our proof is organized around a mixed weighted estimate designed
together with the final time truncation. We use the critical spatial
input of Escauriaza and Vega in the parametrix formulation of
Koch and Tataru~\cite[Proposition~5.1]{KT2009}. For the quadratic
weight
\begin{align*}
  h(s)=as+\frac A2s^2,
\end{align*}
we combine three components in the solution and source norms:
the temporal endpoints $L^\infty_sL^2_y$ and $L^1_sL^2_y$, the
critical Sobolev pair $L^2_sL^{\frac{2n}{n-2}}_y$ and
$L^2_sL^{\frac{2n}{n+2}}_y$, and $L^2_{s,y}$ terms with respective
factors $A^{1/4}$ and~$A^{-1/4}$. Following the parametrix,
coercivity, and duality construction in~\cite{KT2009},
Proposition~\ref{P01} establishes this estimate with constants
uniform in the convexity and slope parameters under the stated
conditions. This combination keeps the time endpoints while
separating the effects of critical integrability and the size of a
bounded potential.

To obtain the full spatial norm, we apply the estimate at every
physical-space center and integrate over the centers. The identity
\begin{align*}
  \int_{\R^n}\int_{\R^n}
  e^{-|y|^2/4}|F(z+\sqrt\rho\,y)|^2
  \,\text{d}y\,\text{d}z
  =(4\pi)^{n/2}\|F\|_{L^2_x}^2,
\end{align*}
for $F\in L^2(\R^n)$ and $\rho>0$, removes the Gaussian weight
exactly. Minkowski's inequality preserves the $L^1$ time norm of
the source, so a time cutoff $\chi$ costs only
$\int|\chi'(t)|\,\text{d}t$, independently of the transition
lengths. For a comparison at $t_1<t_2<t_3$, extending the solution
past $t_3$ by the free heat flow
controls the final cutoff by $\|u(t_3)\|_{L^2_x}$; the slope $a$
is then chosen from the earlier norms to absorb the initial cutoff.
Our use of time extension is inspired by
Huang and Xu~\cite{HuangXu2026}. The estimate is formulated directly
for energy solutions, and continuity into $L^2_x$ permits evaluation
at the prescribed times.

At the small critical endpoint, the critical Sobolev component
absorbs $vu$ with fixed convexity, and the slope produces the linear
time exponent. For a general subcritical potential, we decompose
\begin{align*}
  v=v_c+v_b,
  \wherewhere  
  \|v_c\|_{L^\infty_tL^{n/2}_x}\leq\epsilon,\qquad
  \|v_b\|_{L^\infty_{t,x}}\leq R.
\end{align*}
The same critical estimate absorbs $v_cu$, while the explicit
$L^2$ gain absorbs $v_bu$ once
\begin{align*}
  A\geq C\left(1+R^2(t_3-t_1)^2\right).
\end{align*}
Thus the slope compares the solution norms and the convexity
controls the bounded part of the potential. This gives the
quadratic time dependence in one comparison over the entire
interval, and explains why the critical endpoint is a separate
argument rather than a limit as $p$ decreases to~$n/2$.

The sharpness construction has a different source. Meshkov's
work~\cite{Meshkov} contains both the elliptic example with spatial
decay $e^{-c|x|^{4/3}}$ and, in Section~3, a parabolic example on
$\T^3$ with decay $e^{-ct^2}$;
see also~\cite{CKW} for the sharp estimate on the order of vanishing
for elliptic equations.

The elliptic construction also enters
the sharp dependence of observability constants on potentials;
see Duyckaerts, Zhang, and Zuazua~\cite{DuyckaertsZhangZuazua2008}.
We use Meshkov's periodic parabolic construction, whose successive
transfers to higher Fourier modes generate Gaussian time decay,
and realize it in the global energy class on~$\R^3$.

The additional issue is to localize the periodic profile without
creating singularities in the potential at its zeros. If
$(\partial_t-\Delta)W=qW$ and $\tilde u=gW$, then
\begin{align*}
  (\partial_t-\Delta)\tilde u
  =q\tilde u+W(\partial_tg-\Delta g)-2\nabla g\cdot\nabla W.
\end{align*}
A general decaying envelope need not control the last term
near cancellations of~$W$. We construct a positive, slowly
expanding envelope $g$ with uniformly bounded $L^2$-norm, whose
derivatives in the cancellation directions vanish at the required
rate. In particular,
\begin{align*}
  |\partial_tg|+|\Delta g|\leq Cg
  \andand
  |\nabla g\cdot\nabla W|\leq Cg|W|.
\end{align*}
These estimates preserve boundedness of the quotient potential
across every zero of $W$, while the envelope preserves Gaussian
time decay and gives spatial square integrability. Consequently,
Theorem~\ref{T03} yields a nonzero global energy solution on $\R^3$
with $0<\|u(t)\|_{L^2_x}\leq Ce^{-ct^2}$ for $t\geq0$, proving
optimality of the quadratic time exponent for bounded complex
potentials in Euclidean space.

\colb

\section{Main Result}\label{sec02}
Throughout the paper, solutions and potentials may be real- or complex-valued. 

\cole
\begin{Theorem}[Subcritical Gaussian-type lower bound]
\label{T01}
Let $n\geq 3$, $p\in(n/2,\infty]$, and $T_0<T_*$. Suppose
    \begin{equation}
    v\in L^\infty\left((T_0,T_*);L^p(\R^n)\right)
    ,
   \label{EQ18}
     \end{equation}
and set
$
M=\|v\|_{L^\infty_t L^p_x}
$.
Assume that 
$
    u\in C\left([T_0,T_*];L^2(\R^n)\right)\cap L^2\left((T_0,T_*);H^1(\R^n)\right)
$
satisfies
    \begin{align}
        \label{EQ01}
        \partial_t u-\Delta u=v(t,x) u
      .
    \end{align}
There exists a sufficiently small constant $\epsilon\in(0,1)$, 
depending only on $n$, with the following property.
Fix $T_0\leq t_1<t_2<t_3\leq T_*$. If $\|u(t_i)\|_{L^2_x}>0$ for $i=1$, $2$, 
then
\begin{align}
    \|u(t_3)\|_{L^2_x}\geq \|u(t_2)\|_{L^2_x}\exp\left(
    -C\left(
    R^2(t_3-t_1)^2+\frac{t_3-t_1}{t_2-t_1}\left(1+\log_+\frac{\|u(t_1)\|_{L^2_x}}
    {\|u(t_2)\|_{L^2_x}} \right)+1
    \right)
    \right),
\end{align}
where 
  \begin{align}
  \begin{split}
   R
   =
       \begin{cases}
       \epsilon^{-\frac{n}{2p-n}}M^{\frac{2p}{2p-n}},
       & p<\infty
       \\
       M
              &   p=\infty.
    \end{cases}
  \end{split}
   \label{EQ17}
  \end{align}
\end{Theorem}
\colb

Above, we have denoted
$
\log_+ r=\max\{\log r,0\}
$ for $r>0$.
Here and below, the symbols $C$ and~$c$ denote sufficiently large and
small positive constants, respectively, which may change from line to line.
We also henceforth allow all constants to depend on $n$ without mention.

Since $R^2$ is a fixed multiple, depending only on $n$ and $p$, of $M^{\frac{4p}{2p-n}}$ when
$p<\infty$, Theorem~\ref{T01} yields the following large-time bound: 
\[
\|u(t)\|_{L^2_x}\geq c\exp\left(-C\left(M^\frac{4p}{2p-n}(t-t_1)^2
+(t-t_1)(t_2-t_1)^{-1}\right)\right),
\]
where $c,C>0$ depend on $\|u(t_i)\|_{L^2_x}$ for $i=1,2$, $n$, $p$, and $t_2-t_1$, but not on~$t$.
When $p=\infty$, the corresponding first term in the exponent is $M^2(t-t_1)^2$.

\cole
\begin{Theorem}
[Lower bound for small critical potentials]
    \label{T02}
Let $u$ be as in the previous theorem, and assume that $v$ satisfies
\eqref{EQ18} with $p=n/2$.
Then there exists
a sufficiently small constant $\epsilon_n\in(0,1)$ such that
if
    \[
    \|v\|_{L^\infty_tL^{n/2}_x}\leq \epsilon_n,
    \]
then, under the condition $\|u(t_i)\|_{L^2_x}>0$ for $i=1,2$, we have
    \begin{align}
        \|u(t_3)\|_{L^2_x}\geq \|u(t_2)\|_{L^2_x}\exp
        \left(-C\frac{t_3-t_1}{t_2-t_1}
        \left(
        1+\log \frac{\|u(t_1)\|_{L^2_x}}
        {\|u(t_2)\|_{L^2_x}}
        \right)
        \right).
        \label{EQ03}
    \end{align}
\end{Theorem}
\colb

\begin{Remark}
\label{R01}
{\rm
If only $\|u(t_1)\|_{L^2_x}$ is prescribed instead of $\|u(t_i)\|_{L^2_x}$, for $i=1,2$, one cannot in general obtain a uniform lower bound for $\|u(t)\|_{L^2_x}$ at later times. This
obstruction is already present for the free heat equation. Indeed,
consider
\[
\partial_tu-\Delta u=0
\withwith
\|u(0)\|_{L^2_x}=c_0>0
.
\]
Taking the Fourier transform gives
\[
\widehat{u}(t,\xi)=e^{-t|\xi|^2}\widehat{u}(0,\xi)\comma t\geq 0.
\]
Assume that
\begin{align*}
\operatorname{supp}\widehat{u}(0,\cdot)
\subset
\left\{\xi\in\mathbb{R}^n:|\xi|\geq N\right\},
\end{align*}
where $N\in\mathbb{N}_+$. Then Plancherel's theorem yields
\[
\|u(t)\|_{L^2_x}\leq e^{-tN^2}\|u(0)\|_{L^2_x}=c_0 e^{-t N^2 }\comma t>0
\]
Since $N$ can be chosen arbitrarily large while
$\|u(0)\|_{L^2_x}=c_0$ remains fixed, no positive lower bound for
$\|u(t)\|_{L^2_x}$ can depend only on the $L^2_x$ norm at a single
time slice. In this sense, the use of the two time-slice quantities
$\|u(t_i)\|_{L^2_x}$, for $i=1,2$, in Theorem~\ref{T01} is essential.
}
\end{Remark}

\cole
\begin{Corollary}[Global decay lower bounds]
\label{C01}
Suppose that $u$ is a global solution of \eqref{EQ01} on
$[T_0,\infty)$ such that
\[
\|u(T_0)\|_{L^2_x}>0
.
\]
If $p\in(n/2,\infty]$ and
\[
v\in L^\infty\left((T_0,\infty);L^p(\R^n)\right),
\]
then there exist constants $c,C>0$, independent of $t$, such that
\begin{align}
\|u(t)\|_{L^2_x}
\geq
c\exp\left(-C t^2\right)
\end{align}
for $t\geq T_0+1$.
If $p=n/2$ and
\begin{align}
\|v\|_{L^\infty((T_0,\infty);L^{n/2}(\R^n))}
\leq\epsilon_n,
\end{align}
where $\epsilon_n$ is as in Theorem~\ref{T02}, then there exist constants $c,C>0$, independent of $t$, such that
\[
\|u(t)\|_{L^2_x}
\geq
c\exp\left(-C t\right)
\]
for $t\geq T_0+1$.
\label{EQ20}
\end{Corollary}
\colb

\section{The Carleman estimate}
\label{sec03}
\subsection{The frozen Hermite parametrix}
\label{Sec0301}
Consider the harmonic oscillator
\begin{align}
    \label{EQ04}
    \mathcal{H}=-\Delta_y+\frac{|y|^2}{16} \onon{L^2(\R^n)},
\end{align}
whose spectrum is 
\[
\sigma(\mathcal{H})=\left\{\frac{n}{4}+\frac{k}{2}:k\in \mathbb{N}_0\right\}.
\]
After a fixed normalization of \cite[(5.1)-(5.3)]{KT2009}, we obtain the following endpoint parametrix estimate.

\cole
\begin{Lemma}[Frozen endpoint parametrix]
\label{L01}
There exist constants $\tau_0,C_0>0$
such that
whenever
    \[
    \tau\geq \tau_0
    \andand
    \operatorname{dist}(4\tau,\mathbb{Z})\geq \frac{1}{4},
    \]
    the operator
    \[
    L_\tau =\partial_s + \mathcal{H}-\tau
    \]
has a parametrix $K_\tau$, initially defined on
$C_c^\infty(\R^{1+n})$,
and by extending it to a bounded operator from
    \[
    L^1_sL^2_y+L^2_sL^{\frac{2n}{n+2}}_y
    \quad\text{to}\quad
    L^\infty_sL^2_y\cap L^2_sL^{\frac{2n}{n-2}}_y,
    \]
it satisfies $L_\tau K_\tau f=f$ in the sense of distributions and
    $K_\tau L_\tau w=w$ for $w\in C_c^\infty(\R^{1+n})$ and obeys
    \begin{align}
        \label{EQ05}
        \|K_\tau f\|_{L^\infty_s L^2_y}+\|K_\tau f\|_{L^2_sL^{\frac{2n}{n-2}}_y}\leq C_0
        \inf_{f=f_1+f_2}\left(\|f_1\|_{L^1_sL^2_y}+\|f_2\|_{L^2_sL^{\frac{2n}{n+2}}_y}\right).
    \end{align}
    The same estimate holds for the adjoint operator $L^*_\tau=-\partial_s+\mathcal{H}-\tau$.
\end{Lemma}
\colb

This lemma is a rescaled version of~\cite[Proposition~5.1]{KT2009}.

\subsection{The convexified mixed estimate}
\label{sec0302}
We now patch the frozen parametrices from the previous section to derive a mixed Carleman estimate for the convexified weight
\[
h(s)=as+\frac{A}{2}s^2.
\]
For $A\geq 1$, define the weighted norms
\begin{align}
    \begin{split}
    &\|g\|_{X_A}:=\|g\|_{L^\infty_s L^2_y}+\|g\|_{L^2_sL^\frac{2n}{n-2}_y}+A^\frac{1}{4}\|g\|_{L^2_{s,y}},
    \end{split}
   \llabel{EQ19}
    \end{align}
and    
   \begin{align}
    \begin{split}
    \|f\|_{Y_A}:=\inf_{f=f_1+f_2+f_A}
    \left(
    \|f_1\|_{L^1_sL^2_y}+\|f_2\|_{L^2_sL^\frac{2n}{n+2}_y}
    +A^{-\frac{1}{4}}\|f_A\|_{L^2_{s,y}}
    \right).
    \end{split}
   \llabel{EQ29}
   \end{align}
Once the estimate below is proven with the exponent $1/4$, it also holds
with any $\alpha\in(0,1/4]$ in place of $1/4$ by monotonicity of the corresponding norms.
However, for simplicity, we use the endpoint value $1/4$ throughout.

\cole
\begin{Proposition}[Convexified mixed estimate]
    \label{P01}
Let $I_{\rin}\Subset I_{\out}$ be fixed bounded intervals. There exists
$A_0\geq1$, depending only on $I_{\rin}$ and $I_{\out}$, such that the
following holds. For every $A\geq A_0$ and $a\in\R$ such that
    \begin{align}
    \label{EQ06}
    a+As\geq \tau_0+1
\comma
    s\in I_{\out},
   \end{align}
set
    \[
    h(s)=as+\frac{A}{2}s^2
    \onon{\mathbb{R}}.
    \]
    Define
    \[
    L_h=\partial_s+\mathcal{H}-h'(s).
    \]
    Suppose
    \[
    g\in L^2(\R^{1+n}),
    \qquad
    \supp_s g\Subset I_{\rin},
    \]
and, in the sense of distributions,
    \[
    L_hg\in Y_A.
    \]
Then $g\in X_A$ and
    \begin{align}
        \label{EQ07}
        \|g\|_{X_A}\leq C_{I_{\rin},I_{\out}}\|L_h g\|_{Y_A}.
    \end{align}
The constant depends only on the two interval lengths and
    $\operatorname{dist}(I_{\rin},\R\setminus I_{\out})$; it is independent
    of $A$, $a$, and a translation of the two intervals. 
\end{Proposition}
\colb

In all applications below, the intervals are the fixed intervals chosen in
Proposition~\ref{P02}. Thus, we denote the constant in \eqref{EQ07} simply by~$C$.

\begin{proof}[Proof of Proposition~\ref{P01}]
Fix intermediate intervals
    \[
    I_{\rin}\Subset I_0 \Subset I_1\Subset I_{\out}
    \]
    with fixed relative geometry. Choose a nonnegative smooth partition of unity
$\{\eta_j\}_{j\in\mathbb Z}$ on $I_{\rin}$ and cutoffs
$\{\chi_j\}_{j\in\mathbb Z}$ supported in $I_0$ such
that
\begin{enumerate}
\item $\sum_j\eta_j=1$ on $I_{\rin}$;
\item       $\supp\eta_j\subset[s_j-A^{-\frac{1}{2}},s_j+A^{-\frac{1}{2}}]$,
for some $s_j\in\R$;
\item $\chi_j=1$ on $\supp\eta_j$, the support of $\chi_j$ is contained in an interval
      $[s_j-3A^{-\frac{1}{2}},s_j+3A^{-\frac{1}{2}}]$, and
      $|\chi_j'|\leq CA^{\frac{1}{2}}$;
\item 
both families of supports have uniformly bounded overlap, i.e.
there exists an absolute constant $N$, independent of $A$, such that every $s\in I_{0}$ belongs to at most $N$ supports in each of the two families.
\end{enumerate}
Now, choose $A_0$ sufficiently large so that $A^{-\frac{1}{2}}\leq \frac{1}
{10}\dist(I_{\rin},\partial I_0)$, and  $\tau_j$ such that
$4\tau_j$ is a half-integer
nearest to~$4h'(s_j)$.  Then,
\begin{equation}
  |\tau_j-h'(s_j)|\leq \frac{1}{8}
  \andand
  \dist(4\tau_j,\mathbb Z)=\frac{1}{2}.
  \label{EQ08}
\end{equation}
The condition \eqref{EQ06} ensures that $\tau_j\geq\tau_0$. 
When computing the $Y_A$ norm of $L_hg$,
we may assume that all terms in the decomposition are supported in~$I_{\rin}$.
Indeed, since $L_hg$ is supported in $I_{\rin}$,
multiplying each term by $\mathbf{1}_{I_{\rin}}$ preserves the sum and does not increase any of the three norms.
For $f_2\in L^2_sL^{\frac{2n}{n+2}}_y$, define
\begin{align}
    T_2 f_2=\sum_j\chi_j K_{\tau_j}(\eta_j f_2).
\end{align}
The finite overlap property gives 
\begin{align}
\label{EQ12}
\sum_j \|\eta_j f_2\|_{L^2_s L^\frac{2n}{n+2}_y}^2\leq C \|f_2\|^2_{L^2_sL^\frac{2n}{n+2}_y}.   
\end{align}
Writing $u_j=K_{\tau_j}(\eta_j f_2)$, Lemma~\ref{L01} gives
\[
\|u_j\|_{L^\infty_sL^2_y}+\|u_j\|_{L^2_sL^{\frac{2n}{n-2}}_y}
\leq C \|\eta_jf_2\|_{L^2_sL^\frac{2n}{n+2}_y}.
\]
We now estimate the three components of~$\|T_2f_2\|_{X_A}$.
At every time $s$, at most a fixed number of $\chi_j$ are nonzero. Hence, the Cauchy-Schwarz inequality gives
\begin{align}
\label{EQ09}
    \begin{split}
    &\left\|\sum_j \chi_j u_j\right\|^2_{L^\infty_sL^2_y}
    +\left\|\sum_j \chi_j u_j\right\|^2_{L^2_sL^\frac{2n}{n-2}_y}
    \leq C\sum_j
    \|u_j\|^2_{L^\infty_sL^2_y}
    +C\sum_j
    \|u_j\|^2_{L^2_sL^\frac{2n}{n-2}_y}
    \\&\indeq
    \leq C \sum_{j}
    \|\eta_j f_2\|_{L^2_s L^\frac{2n}{n+2}_y}^2.
    \end{split}
\end{align}
Moreover, the definition of cut-off together with H\"older's inequality implies 
\begin{align}
\label{EQ10}
\begin{split}
A^{\frac{1}{2}}\left\|\sum_j\chi_j u_j\right\|_{L^2_{s,y}}^2
&\leq C\sum_j
\underbrace{A^{\frac{1}{2}}\|\chi_j\|^2_{L^2_sL^\infty_y}}
_{\leq A^{\frac{1}{2}}|\operatorname{supp}\chi_j|\leq C}
\|u_j\|^2_{L^\infty_sL^2_y}\\
&\leq C\sum_j
\|\eta_jf_2\|_{L^2_sL^\frac{2n}{n+2}_y}^2
\leq C\|f_2\|_{L^2_sL^\frac{2n}{n+2}_y}^2.
\end{split}
\end{align}
Combining~\eqref{EQ09} and~\eqref{EQ10}, we obtain
\begin{align}
        \label{EQ11}
        \|T_2f_2\|_{X_A}\leq C \|f_2\|_{L^2_sL^\frac{2n}{n+2}_y}.
\end{align}
Since $L_{\tau_j}K_{\tau_j}=I$ and $\chi_j=1$ on $\operatorname{supp}\eta_j$, we derive
\begin{align*}
\begin{split}
(L_hT_2-I)f_2=\sum_j(L_{\tau_j}+(\tau_j-h'))(\chi_jK_{\tau_j}(\eta_j f_2))-f_2
=\sum_j(\chi'_j+(\tau_j-h')\chi_j)K_{\tau_j}(\eta_jf_2).    
\end{split}
\end{align*}
On $\operatorname{supp}\chi_j$, we have
\[
|h'(s)-\tau_j|+|\chi'_j(s)|\leq \underbrace{|h'(s)-h'(s_j)|}_{= A|s-s_j|}+|h'(s_j)-\tau_j|+|\chi'_j(s)|\leq 
C(A^{-\frac{1}{2}} A+1+A^\frac{1}{2})\leq CA^\frac{1}{2},
\]
since $A\geq A_0>1$.
Also, H\"older's inequality on $\operatorname{supp}\chi_j$, Minkowski's inequality,  and~\eqref{EQ12} give
\begin{align*}
    \begin{split}
    &
    \left\|\sum_j\mathbf{1}_{\operatorname{supp}\chi_j}K_{\tau_j}(\eta_j f_2)\right\|_{L^2_{s,y}}^2\leq 
    C
    \sum_j
    \left\|\mathbf{1}_{\operatorname{supp}\chi_j}K_{\tau_j}(\eta_jf_2)\right\|_{L^2_{s,y}}^2
    \\&\indeq
    \leq
    C
    \sum_j
    \|\mathbf{1}_{\operatorname{supp}\chi_j}\|^2_{L^2_sL^\infty_y}
    \|\underbrace{K_{\tau_j}(\eta_jf_2)}_{=u_j}\|^2_{L^\infty_sL^2_y}
    \\&\indeq
    \leq C A^{-\frac{1}{2}}\sum_j \|\eta_j f_2\|_{L^2_s L^\frac{2n}{n+2}_y}^2\leq C A^{-\frac{1}{2}}\|f_2\|_{L^2_sL^\frac{2n}{n+2}_y}^2
    ,
    \end{split}
\end{align*}
from where we get
\begin{align}
\label{EQ13}
\begin{split}
A^{-\frac{1}{4}}\|(L_hT_2-I)f_2\|_{L^2_{s,y}}
&\leq CA^{-\frac{1}{4}}
\left(\sum_j
\|(\chi'_j+(\tau_j-h')\chi_j)u_j\|_{L^2_{s,y}}^2
\right)^{\frac12}\\
&\leq CA^{\frac14}
\left(\sum_j
\|\mathbf{1}_{\operatorname{supp}\chi_j}u_j\|_{L^2_{s,y}}^2
\right)^{\frac12}
\leq C\|f_2\|_{L^2_sL^{\frac{2n}{n+2}}_y}.
\end{split}
\end{align}
For $f_1\in L^1_sL^2_y$, define
\[
T_1f_1=\sum_j\chi_jK_{\tau_j}(\eta_jf_1).
\]
Since $\eta_j\geq0$ and $\sum_j\eta_j=1$ on $\operatorname{supp}f_1$,
we have
\begin{align*}
\sum_j\|\eta_jf_1\|_{L_s^1L_y^2}
&=\int_{\R}\sum_j\eta_j(s)\|f_1(s)\|_{L_y^2}\,\text{d}s
=\|f_1\|_{L_s^1L_y^2}
\end{align*}
and
\begin{align*}
\sum_j\|\eta_jf_1\|_{L_s^1L_y^2}^2
&\leq\left(\sum_j\|\eta_jf_1\|_{L_s^1L_y^2}\right)^2
.
\end{align*}
The same calculations leading to \eqref{EQ11} and \eqref{EQ13} then give
\begin{align} 
\label{EQ14}
        \|T_1f_1\|_{X_A}\leq C \|f_1\|_{L_s^1L^2_y},
\end{align} and
\begin{align} 
\label{EQ15}
A^{-\frac{1}{4}}\|(L_hT_1-I)f_1\|_{L^2_{s,y}}\leq C \|f_1\|_{L_s^1L^2_y}.
\end{align}    
Write $L_h g=f_1+f_2+f_A$, and set
\[
w=g-T_1f_1-T_2f_2.
\]
Inequalities~\eqref{EQ13} and \eqref{EQ15} imply
\begin{align}
    \label{EQ16}
    \begin{split}
    &A^{-\frac{1}{4}}\|L_hw\|_{L^2_{s,y}}
    \leq \sum_{i=1,2}A^{-\frac{1}{4}}\|(L_hT_i-I)f_i\|_{L^2_{s,y}}+A^{-\frac{1}{4}}\|L_hg-f_1-f_2\|_{L^2_{s,y}}
    \\&\indeq
    \leq C\left(
    \|f_1\|_{L^1_sL^2_y}+\|f_2\|_{L^2_sL^\frac{2n}{n+2}_y}+A^{-\frac{1}{4}}\|f_A\|_{L^2_{s,y}}
    \right).
    \end{split}
\end{align}
Choose $\zeta\in C_c^\infty(I_1)$ with $0\leq\zeta\leq1$ and $\zeta=1$
on~$I_0$.  Since $w$ is supported in $I_0$, we have
\[
  \langle w,\phi\rangle=\langle w,\zeta\phi\rangle
\]
for every
$\phi\in C_c^\infty(\R^{1+n})$.
Moreover, multiplication by $\zeta$ is bounded on both
$L^1_sL^2_y$ and $L^2_sL^{\frac{2n}{n+2}}_y$.  
It therefore suffices to consider test functions supported in~$I_1$.
On $I_1$, we repeat the preceding patched-parametrix construction for the adjoint frozen operators
\[
L_\tau^*=-\partial_s+\mathcal{H}-\tau
\]
using cutoffs supported in~$I_{\out}$.
Denote the resulting patched right parametrices by $S_1$ and~$S_2$.  The same calculations give
\begin{align}
 A^\frac{1}{4}\|S_1\phi\|_{L^2_{s,y}}
+A^{-\frac{1}{4}}\|(L_h^*S_1-\operatorname{Id})\phi\|_{L^2_{s,y}}
&\leq C\|\phi\|_{L^1_sL^2_y},
\label{EQ24}
\end{align}
and
\begin{align}
 A^\frac{1}{4}\|S_2\phi\|_{L^2_{s,y}}
+A^{-\frac{1}{4}}\|(L_h^*S_2-\operatorname{Id})\phi\|_{L^2_{s,y}}
&\leq C\|\phi\|_{L^2_sL^{\frac{2n}{n+2}}_y}.
\label{EQ25}
\end{align}
Let $\phi\in C_c^\infty(I_1\times\R^n)$, and let $S$ be either $S_1$
or~$S_2$. Then $S\phi$ has compact time support, belongs to $L^2$,
and~$L_h^*S\phi\in L^2$.  Therefore, the graph Green identity in
Lemma~\ref{L02} applies and gives
\begin{align*}
|\langle w,\phi\rangle|
&=|\langle w,L_h^*S\phi\rangle
 -\langle w,(L_h^*S-\operatorname{Id})\phi\rangle|\\
&\leq
\|A^{-\frac{1}{4}}L_hw\|_{L^2_{s,y}}
\,\|A^\frac{1}{4}S\phi\|_{L^2_{s,y}}
+\|A^{\frac{1}{4}}w\|_{L^2_{s,y}}\,
\|A^{-\frac{1}{4}}(L_h^*S-\operatorname{Id})\phi\|_{L^2_{s,y}}.
\end{align*}
Applying \eqref{EQ24} and \eqref{EQ25}, respectively, and then taking the
supremum over the corresponding dense test classes, we obtain from the
standard duality identities
$
  (L^1_sL^2_y)^*=L^\infty_sL^2_y
$ and
$  \left(L^2_sL^{\frac{2n}{n+2}}_y\right)^*
  =L^2_sL^{\frac{2n}{n-2}}_y
$ 
that
\begin{equation}
\|w\|_{L^\infty_sL^2_y}+\|w\|_{L^2_sL^\frac{2n}{n-2}_y}
\leq C\left(A^\frac{1}{4}\|w\|_2+A^{-\frac{1}{4}}\|L_hw\|_2\right).
\label{EQ27}
\end{equation}
Combining \eqref{EQ27} and \eqref{EQ21}, we obtain
\begin{equation}
\|w\|_{X_A}\leq CA^{-\frac{1}{4}}\|L_hw\|_2.
\label{EQ26}
\end{equation}
Equations \eqref{EQ14}, \eqref{EQ11}, \eqref{EQ21} in Lemma~\ref{L02}, and \eqref{EQ26} imply
\[
\|g\|_{X_A}
\leq C\left(
\|f_1\|_{L^1_sL^2_y}
+\|f_2\|_{L^2_sL^{\frac{2n}{n+2}}_y}
+A^{-\frac{1}{4}}\|f_A\|_2
\right).
\]
Taking the infimum over all admissible decompositions in the definition of the $Y_A$-norm then proves~\eqref{EQ07}.
\end{proof}

\subsection{The physical-space estimate}
We now transfer the mixed Carleman estimate from the similarity variables back to physical space.

\cole
\begin{Proposition}[Physical-space weighted estimate in the energy class]
\label{P02}
Let $J\subset\R$ be a bounded open interval, let $T_1>\sup J$, and let $s_1>0$.  For $t<T_1$, set
\[
\sigma(t)=\frac{T_1-t}{s_1}
\andand
s(t)=-\log\sigma(t)
      =\log\frac{s_1}{T_1-t}.
\]
Assume
\[
V=V_{\crit}+V_{\bdd},
\wherewhere
\|V_{\crit}\|_{L^\infty(J;L^{n/2}(\R^n))}\leq\epsilon,
\qquad
\|V_{\bdd}\|_{L^\infty(J\times\R^n)}\leq R.
\]
Let
\begin{align}
\phi(s)=as+\frac A2(s+2)^2,
\label{EQ30}
\end{align}
where
\begin{equation}
A\geq A_n\left(1+(s_1R)^2\right)\comma
a\geq a_n,
   \label{EQ45}
     \end{equation}
and
where $A_n,a_n$, and $\epsilon>0$ depend only on~$n$.
Suppose
\[
F\in C_c(J;L^2(\R^n))\cap L^2(J;H^1(\R^n)),
\andand
G\in L^1(J;L^2(\R^n))
\]
satisfies
\[
(\partial_t-\Delta)F=VF+G
\]
in $\mathcal{D}'(J\times\R^n)$. Assume further that
\[
\supp_tF\subset\{t\in J:0\leq s(t)\leq\log4\}.
\]
Then
\begin{align}
e^{\phi(s(t_*))}\|F(t_*)\|_{L^2_x}
\leq C\int_J e^{\phi(s(t))}\|G(t)\|_{L^2_x}\,dt
\label{EQ31}
\end{align}
for every $t_*\in J$.
\end{Proposition}
\colb

\begin{proof}[Proof of Proposition~\ref{P02}]
Since $F$ has compact time support in $J$, it vanishes on a neighborhood of~$\partial J$.  We may therefore extend $F$, $G$, and the two parts of $V$ by zero outside $J$; the extended equation still holds distributionally, with no additional time-boundary term.
Fix the intervals
\[
I_{\rin}=\left(-\frac12,\log4+\frac12\right)
\andand
I_{\out}=(-1,\log4+1).
\]
Choose $A_n$ large enough that $A\geq A_0(n,I_{\rin},I_{\out})$ whenever \eqref{EQ45} holds.

\medskip
\noindent\textbf{Step 1: the similarity transform.}
For $s\in\R$, define the backward time distance
\[
\rho(s)=s_1e^{-s}
\andand
t(s)=T_1-\rho(s).
\]
Thus $\text{d}t/\text{d}s=\rho(s)$.  For each physical-space center $z\in\R^n$, set
\[
x_z(s,y)=z+\sqrt{\rho(s)}\,y
\]
and define
\begin{equation}
g_z(s,y)
=e^{\phi(s)}e^{-|y|^2/8}
 F\left(t(s),x_z(s,y)\right).
\label{EQ70}
\end{equation}
The advantage of using a center $z$ in the original $x$-variables is that no separate rescaled spatial variable is needed.
Set
\begin{align}
\Phi(s)=\phi(s)+\frac n4s.
\label{EQ38}    
\end{align}
We claim that
\begin{align}
\left(\partial_s+\mathcal{H}-\Phi'(s)\right)g_z
=\mathcal V_{\crit,z}g_z+\mathcal V_{\bdd,z}g_z+\mathcal G_z
\label{EQ34}    
\end{align}
in the sense of distributions,
where
\begin{align*}
\mathcal V_{\alpha,z}(s,y)
&=\rho(s)V_\alpha\left(t(s),x_z(s,y)\right)
\comma
\alpha\in\{\crit,\bdd\},
\end{align*}
and
\begin{align}
\mathcal G_z(s,y)
&=\rho(s)e^{\phi(s)}e^{-|y|^2/8}
 G\left(t(s),x_z(s,y)\right).
\label{EQ35}
\end{align}
For clarity, we show the whole argument.  If
\[
\tilde F_z(s,y)=F\left(t(s),x_z(s,y)\right),
\]
then
\[
\partial_s\tilde F_z
=\rho\,\partial_tF-\frac12y\cdot\nabla_y\tilde F_z
\andand
\Delta_y\tilde F_z=\rho\,\Delta_xF.
\]
Consequently,
\begin{align}
\rho(\partial_t-\Delta_x)F\left(t(s),x_z(s,y)\right)
=\left(\partial_s+\frac12y\cdot\nabla_y-\Delta_y\right)\tilde F_z.
\label{EQ37}    
\end{align}
The elementary conjugation identity
\[
e^{-|y|^2/8}
\left(\frac12y\cdot\nabla_y-\Delta_y\right)
 e^{|y|^2/8}
=\mathcal{H}-\frac n4
\]
turns \eqref{EQ37} into
\[
\left(\partial_s+\mathcal{H}-\frac n4\right)
\left(e^{-|y|^2/8}\tilde F_z\right)
=\rho e^{-|y|^2/8}(\partial_t-\Delta_x)F.
\]
Multiplication by $e^{\phi(s)}$ then gives~\eqref{EQ34}.
The energy assumptions ensure that $g_z$ belongs to the graph class needed below. Indeed,
\begin{align}
\|g_z(s)\|_{H^1_y}
\leq C e^{\phi(s)}\rho(s)^{-\frac{n}{4}}
\left(
 \|F(t(s))\|_{L^2_x}
 +\sqrt{\rho(s)}\,\|\nabla F(t(s))\|_{L^2_x}
\right),
\label{EQ36}    
\end{align}
with a constant independent of~$z$.  On the support of $F$, one has
\begin{align}
\label{EQ50}
0\leq s\leq\log4
\andand
\frac{s_1}{4}\leq\rho(s)\leq s_1,
\end{align}
so \eqref{EQ36} implies
\[
g_z\in L^2_sH^1_y
\andand
\supp_sg_z\subset[0,\log4]\Subset I_{\rin}.
\]
In particular, Sobolev embedding already gives $g_z\in L^2_sL^\frac{2n}{n-2}_y$ independently of the mixed estimate. Moreover, \eqref{EQ35} and $G\in L^1_tL^2_x$ imply
\[
\mathcal G_z\in L^1_sL^2_y.
\]
Thus every term in \eqref{EQ34} belongs to the appropriate source space.

\medskip
\noindent\textbf{Step 2: apply the mixed estimate and absorb the potential.}
By \eqref{EQ38}, we have
\begin{align}
\Phi(s)
=\frac A2s^2+\left(a+2A+\frac n4\right)s+2A.
\label{EQ39}    
\end{align}
The additive constant $2A$ is irrelevant to the operator, and \eqref{EQ39} has exactly the form required in Proposition~\ref{P01}.  Since $s+2\geq1$ on $I_{\out}$, choosing $a_n$ sufficiently large guarantees the slope condition for~$\Phi$.
For fixed $s$ and $z$, the change of variables
\[
x=z+\sqrt{\rho(s)}\,y
\]
gives the critical scaling identity
\begin{align}
\|\mathcal V_{\crit,z}(s)\|_{L^{n/2}_y}
=\|V_{\crit}(t(s))\|_{L^{n/2}_x}
\leq\epsilon.
\label{EQ32}    
\end{align}
The equation~\eqref{EQ50} then gives
\begin{align}
\|\mathcal V_{\bdd,z}\|_{L^\infty(\operatorname{supp}_sg_z\times\R^n)}
\leq s_1R.
\label{EQ33}
\end{align}
Then apply Proposition~\ref{P01} to \eqref{EQ34} using
\[
\mathcal G_z\in L^1_sL^2_y,
\qquad
\mathcal V_{\crit,z}g_z\in L^2_sL^{\frac{2n}{n+2}}_y,
\andandone
\mathcal V_{\bdd,z}g_z\in L^2_{s,y}.
\]
Hölder's inequality, \eqref{EQ32}, and \eqref{EQ33} yield
\begin{align*}
\|g_z\|_{X_A}
\leq C\Bigl(&
 \|\mathcal G_z\|_{L^1_sL^2_y}
 +\epsilon\|g_z\|_{L^2_sL^\frac{2n}{n-2}_y}
 +A^{-\frac{1}{4}}s_1R\|g_z\|_{L^2_{s,y}}
\Bigr).
\end{align*}
Choosing $\epsilon$ sufficiently small and $A_n$ sufficiently large, the last two terms are then absorbed into the left-hand side, uniformly in the center~$z$.  We obtain
\begin{align}
\|g_z\|_{L^\infty_sL^2_y}
\leq C\|\mathcal G_z\|_{L^1_sL^2_y}.
\label{EQ40}
\end{align}
For fixed $z$, the map $s\mapsto g_z(s)$ is continuous into $L^2_y$,
which follows from $F\in C_tL^2_x$ and by the continuity on $L^2$ of
translations, positive dilations, and multiplication by the fixed Gaussian. Hence, the essential $L^\infty_sL^2_y$ estimate in \eqref{EQ40} holds at every prescribed time $s_*=s(t_*)$.  If $F(t_*)=0$, the desired conclusion is immediate; otherwise, $s_*\in[0,\log4]$.

\medskip
\noindent\textbf{Step 3: average over the physical-space centers.}
Averaging over the centers is based on the following exact identity: For every $Q\in L^2(\R^n)$ and $\rho>0$,
\begin{align}
\int_{\R^n_z}
 \|e^{-|y|^2/8}Q(z+\sqrt\rho\,y)\|_{L^2_y}^2\,\text{d}z
=(4\pi)^{n/2}\|Q\|_{L^2_x}^2.
\label{EQ41}
\end{align}
Indeed, Tonelli's theorem and the substitution $x=z+\sqrt\rho\,y$ give
\begin{align*}
&\int_{\R^n_z}\int_{\R^n_y} e^{-|y|^2/4}|Q(z+\sqrt\rho\,y)|^2\,\text{d}y\,\text{d}z
\\&\indeq
=\int_{\R^n_y} e^{-|y|^2/4}
   \left(\int_{\R^n_z}|Q(z+\sqrt\rho\,y)|^2\,\text{d}z\right)\text{d}y
=\left(\int_{\R^n_y} e^{-|y|^2/4}\,\text{d}y\right)\|Q\|_2^2.
\end{align*}
Fix $s_*=s(t_*)$. Taking the $L^2$-norm in $z$ in \eqref{EQ40} and applying Minkowski's integral inequality, we obtain
\begin{align}
\left(\int_{\R^n_z}\|g_z(s_*)\|_{L^2_y}^2\,\text{d}z\right)^{\frac{1}{2}}
\leq C\int_\R
\left(\int_{\R^n_z}\|\mathcal G_z(s)\|_{L^2_y}^2\,\text{d}z\right)^{\frac{1}{2}}\text{d}s.
\label{EQ42}
\end{align}
Applying \eqref{EQ41} first to $Q=F(t(s))$ and then to $Q=G(t(s))$ gives
\begin{align}
\left(\int_{\R^n_z}\|g_z(s)\|_{L^2_y}^2\,\text{d}z\right)^{\frac{1}{2}}
=(4\pi)^{\frac{n}{4}}e^{\phi(s)}\|F(t(s))\|_{L^2_x},
\label{EQ43}
\end{align}
and
\begin{align}
\left(\int_{\R^n_z}\|\mathcal G_z(s)\|_{L^2_y}^2\,\text{d}z\right)^{\frac{1}{2}}
=(4\pi)^{\frac{n}{4}}\rho(s)e^{\phi(s)}\|G(t(s))\|_{L^2_x}.
\label{EQ44}
\end{align}
Substituting \eqref{EQ43}--\eqref{EQ44} into \eqref{EQ42} and canceling $(4\pi)^{\frac{n}{4}}$ yields
\[
e^{\phi(s_*)}\|F(t_*)\|_2
\leq C\int_\R \rho(s)e^{\phi(s)}\|G(t(s))\|_2\,\text{d}s.
\]
Finally, $\text{d}t=\rho(s)\,\text{d}s$, and the extended source is supported in~$J$.  Therefore, the last integral equals
\[
\int_J e^{\phi(s(t))}\|G(t)\|_2\,\text{d}t,
\]
which proves~\eqref{EQ31}.
\end{proof}

\section{Proof of the lower bounds}
\label{sec04}
\subsection{Truncation of the potential and the energy bound}
\label{sec0401}
We first decompose the potential into a small critical part and a bounded part, and then record the corresponding energy estimate.
Let $S$ be the Sobolev constant in
\[
\|f\|_{L^\frac{2n}{n-2}_x}^2\leq S\|\nabla f\|_{L^2_x}^2.
\]
Decreasing $\epsilon$ if necessary, we assume that
\[
\epsilon S\leq \frac{1}{2}.
\]
Let $M=\|v\|_{L_t^\infty L_x^p}$, and let $R$ be defined
by~\eqref{EQ17}.
Suppose first that $p<\infty$. If $M=0$, set
$v_c=v_b=0$. If $M>0$, set
\[
v_c=v\mathbf{1}_{\{|v|>R\}}
\andand
v_b=v\mathbf{1}_{\{|v|\leq R\}}.
\]
Then, for almost every $t$,
\[
\|v_c(t)\|_{L_x^\frac{n}{2}}^{\frac{n}{2}}
 =\int_{\{|v(t)|>R\}}|v(t,x)|^{\frac{n}{2}}\,\text{d}x
 \leq R^{\frac{n}{2}-p}\|v(t)\|_{L_x^p}^p.
\]
Consequently,
\[
\|v_c\|_{L_t^\infty L_x^{\frac{n}{2}}}
\leq R^{1-\frac{2p}{n}}M^{\frac{2p}{n}}
=\epsilon.
\]
The same conclusion is immediate when $M=0$. Thus, in either case,
\begin{equation}
\|v_c\|_{L_t^\infty L_x^{n/2}}\leq\epsilon
\andand
\|v_b\|_{L_{t,x}^\infty}\leq R.
\label{EQ71}
\end{equation}
When $p=\infty$, we instead take $v_c=0$ and $v_b=v$;
then \eqref{EQ71} holds with $R=M$.
We shall also use the following energy bound. Since
\[
\Delta u\in L_t^2H_x^{-1},
\qquad
v_cu\in L_t^2L_x^{2n/(n+2)}
      \hookrightarrow L_t^2H_x^{-1},
\andandone
v_bu\in L_{t,x}^2
      \hookrightarrow L_t^2H_x^{-1},
\]
the equation \eqref{EQ01} implies $u_t\in L_t^2H_x^{-1}$. The Hilbert space
chain rule therefore shows that
$t\mapsto\|u(t)\|_2^2$ is absolutely continuous. Taking real
parts in the $H^{-1},H^1$ duality pairing gives
\[
\begin{aligned}
\frac{1}{2}\frac{\text{d}}{\text{d}t}\|u(t)\|_{L_x^2}^2+\|\nabla u(t)\|_{L_x^2}^2
&\leq
\|v_c(t)\|_{L^\frac{n}{2}_x}\|u(t)\|_{L^{\frac{2n}{n-2}}_x}^2
 +R\|u(t)\|_{L_x^2}^2
\\
&\leq
\frac{1}{2}\|\nabla u(t)\|_{L_x^2}^2+R\|u(t)\|_{L_x^2}^2
\end{aligned}
\]
for almost every~$t$.
Gronwall's inequality then yields
\begin{equation}
\|u(t)\|_{L_x^2}
\leq e^{R(t-s)}\|u(s)\|_{L^2_x},
\qquad
T_0\leq s\leq t\leq T_*.
\label{EQ72}
\end{equation}
In particular, when $R=0$, the $L^2$-norm is nonincreasing.

\subsection{Proof of the main estimates}
\label{sec0402}

\begin{proof}[Proof of Theorem~\ref{T01}]
Let $v_c$ and $v_b$ be as in 
Section~\ref{sec0401}. For $t\geq t_3$, extend
$u$ by the free heat flow and $v$, $v_c$, and $v_b$ by zero.
Thus,
\begin{align}
u(t)=e^{(t-t_3)\Delta}u(t_3)
\comma
t\geq t_3.
\label{EQ52}
\end{align}
Note that the two one-sided traces agree at $t_3$, so the extension introduces no
distributional term supported at $t=t_3$.
For every $t>t_3$,
Plancherel's theorem gives
\[
\int_{t_3}^{t}\|\nabla u(\tau)\|_{L_x^2}^2\,\text{d}\tau
=
\frac{1}{2}\int_{\R^n}
(1-e^{-2(t-t_3)|\xi|^2})
|\widehat{u(t_3)}(\xi)|^2\,\text{d}\xi
\leq \frac{1}{2}\|u(t_3)\|_{L_x^2}^2.
\]
Thus the continued function remains in the local energy class and satisfies
\begin{align}
\|u(t)\|_{L_x^2}\leq \|u(t_3)\|_{L_x^2}
\comma
t\geq t_3.
\label{EQ53}
\end{align}
Set
\[
\delta=\frac{t_2-t_1}{8},
\]
and choose $\chi\in C_c^\infty(\R)$ with $0\leq\chi\leq1$ such that
\[
\chi=1
\quad\text{on }[t_1+\delta,2t_3-t_1]
\andand
\supp\chi\subset[t_1,3t_3-2t_1],
\]
and such that $\chi$ is nondecreasing on $[t_1,t_1+\delta]$ and
nonincreasing on $[2t_3-t_1,3t_3-2t_1]$. Consequently,
\begin{align}
\int_{t_1}^{t_1+\delta}|\chi'(t)|\,\text{d}t
=
\int_{2t_3-t_1}^{3t_3-2t_1}|\chi'(t)|\,\text{d}t
=1.
\label{EQ54}
\end{align}
Define $F(t)=\chi(t)u(t)$ for $t\geq t_1$ and $F(t)=0$ for $t<t_1$.
Extend $v$, $v_c$, and $v_b$ by zero both for $t<t_1$ and
for $t>t_3$.
Since $\chi(t_1)=0$, the zero extension creates no distribution supported
at $t=t_1$, and
\begin{align}
(\partial_t-\Delta)F=vF+\chi'u
\label{EQ55}
\end{align}
holds in the sense of distributions, where $\chi'u$ is understood to be zero for
$t<t_1$.
We now apply Proposition~\ref{P02} with future time $T_1=4t_3-3t_1$ and scale
$s_1=4(t_3-t_1)$. The corresponding similarity variable is
\[
s(t)=\log\frac{4(t_3-t_1)}{4t_3-3t_1-t}.
\]
Choose a bounded open interval $J$ such that
\[
[t_1,3t_3-2t_1]\Subset J
\andand
\sup J<4t_3-3t_1.
\]
On the support of $F$, one has $0\leq s(t)\leq\log4$. Set
\begin{align}
A=A_n(1+16R^2(t_3-t_1)^2)
\andand
\phi(s)=as+\frac{A}{2}(s+2)^2,
\label{EQ56}
\end{align}
where $a\geq a_n$ shall be chosen below. We may assume $a_n\geq0$.
Since the scale in Proposition~\ref{P02} is $4(t_3-t_1)$, the choice
of $A$ in \eqref{EQ56} satisfies the required lower bound on~$A$.
Observe that we have the following values of the similarity variable:
\begin{align}
s(t_1+\delta)
&=
-\log\left(1-\frac{t_2-t_1}{32(t_3-t_1)}\right),\label{EQ57a}
\\
s(t_2)
&=
-\log\left(1-\frac{t_2-t_1}{4(t_3-t_1)}\right),\label{EQ57b}
\\
s(3t_3-2t_1)
&=\log4.
\label{EQ57}
\end{align}
On $\operatorname{supp}\chi$, both $s(t)$ and $\phi(s(t))$ are increasing in $t$, since
\[
s'(t)=\frac{1}{4t_3-3t_1-t}>0
\andand
\phi'(s)=a+A(s+2)>0
\onon{\operatorname{supp}_t\chi}.
\]
On the first transition interval $t\in [t_1,t_1+\delta]$, the
inequality \eqref{EQ72} gives
\[
\|u(t)\|_{L_x^2}
\leq e^{R(t-t_1)}\|u(t_1)\|_{L_x^2}
\leq e^{R\delta}\|u(t_1)\|_{L_x^2},
\]
while on the second transition interval $t\in [2t_3-t_1,3t_3-2t_1]$, \eqref{EQ53} implies
\[
\|u(t)\|_{L_x^2}\leq\|u(t_3)\|_{L_x^2}.
\]
It follows from \eqref{EQ54} that
\begin{align}
\begin{split}
&\int_J e^{\phi(s(t))}|\chi'(t)|\|u(t)\|_{L_x^2}\,\text{d}t
\\&\indeq
\leq \int_{t_1}^{t_1+\delta}e^{\phi(s(t_1+\delta))}|\chi'(t)|e^{R\delta}\|u(t_1)\|_{L_x^2}
+\int_{2t_3-t_1}^{3t_3-2t_1}e^{\phi(s(3t_3-2t_1))}|\chi'(t)|\|u(t_3)\|_{L^2_x}
\\&\indeq
\leq
e^{\phi(s(t_1+\delta))+R\delta}\|u(t_1)\|_{L_x^2}
+
e^{\phi(\log4)}\|u(t_3)\|_{L_x^2}.
\label{EQ58}
\end{split}
\end{align}
Since $t_2-t_1=8\delta$ and $t_2<t_3$, we have $\chi(t_2)=1$.
Applying Proposition~\ref{P02} to \eqref{EQ55} at $t_*=t_2$ and using
\eqref{EQ58}, we obtain
\begin{align}
e^{\phi(s(t_2))}\|u(t_2)\|_{L_x^2}
\leq
C_*\left(
e^{\phi(s(t_1+\delta))+R\delta}\|u(t_1)\|_{L_x^2}
+
e^{\phi(\log4)}\|u(t_3)\|_{L_x^2}
\right),
\label{EQ59}
\end{align}
where $C_*\geq1$ depends only on~$n$.
We next choose 
\begin{align}
a
=
a_n+
\frac{1}{s(t_2)-s(t_1+\delta)}
\log_+\left(
\frac{2C_*e^{R\delta}\|u(t_1)\|_{L_x^2}}
{\|u(t_2)\|_{L_x^2}}
\right).
\label{EQ61}
\end{align}
Since
\begin{align*}
&
\phi(s(t_2))-\phi(s(t_1+\delta))
\\&\indeq
=
a\left(s(t_2)-s(t_1+\delta)\right)
+\frac{A}{2}\left(s(t_2)-s(t_1+\delta)\right)\left(s(t_2)+s(t_1+\delta)+4\right)
\geq
a\left(s(t_2)-s(t_1+\delta)\right)
\end{align*}
and $xe^{-\log_+x}\leq1$ for every $x>0$, the choice \eqref{EQ61}
implies
\begin{align*}
\begin{split}
&C_*e^{\phi(s(t_1+\delta))-\phi(s(t_2))+R\delta}
\|u(t_1)\|_{L_x^2}
\\&\indeq
\leq C_* e^{R\delta}\|u(t_1)\|_{L^2_x}
\exp\biggl(
\underbrace{-a_n(s(t_2)-s(t_1+\delta))}_{\leq 0}
-\log_+\biggl(
\frac{2C_*e^{R\delta}\|u(t_1)\|_{L_x^2}}
{\|u(t_2)\|_{L_x^2}}
\biggr)\biggr)
\\&\indeq
\leq
\frac{1}{2}\|u(t_2)\|_{L_x^2}.
\end{split}
\end{align*}
Dividing \eqref{EQ59} by $e^{\phi(s(t_2))}$ thus gives
\begin{align}
\|u(t_3)\|_{L_x^2}
\geq
\frac{1}{2C_*}\|u(t_2)\|_{L_x^2}
\exp\left(
-\phi(\log4)+\phi(s(t_2))
\right).
\label{EQ62}
\end{align}
Since $0\leq s(t_2)\leq\log(4/3)$, we have
\begin{align}
\begin{split}
&
\phi(\log4)-\phi(s(t_2))
=
\left(\log4-s(t_2)\right)
\left(
a+\frac{A}{2}\left(\log4+s(t_2)+4\right)
\right)
\\&\indeq
\leq 4\log\left(
4-\frac{t_2-t_1}{t_3-t_1}
\right)
(a+A)
\leq 4\log 4\ (a+A)
.
\label{EQ63}
\end{split}
\end{align}
Since
\[
(-\log(1-r))'=\frac{1}{1-r}\in\left[1,\frac{4}{3}\right]
\comma
0\leq r\leq\frac{1}{4},
\]
the mean value theorem, \eqref{EQ57a}, and \eqref{EQ57b} give
\begin{align}
\frac{7(t_2-t_1)}{32(t_3-t_1)}
&\leq
s(t_2)-s(t_1+\delta)
\leq
\frac{7(t_2-t_1)}{24(t_3-t_1)}
.
\end{align}
Moreover, \eqref{EQ61} implies
\begin{align}
\begin{split}
a
&\leq
a_n+
\frac{32(t_3-t_1)}{7(t_2-t_1)}
\left(
\log(2C_*)+R\delta
+\log_+\frac{\|u(t_1)\|_{L_x^2}}{\|u(t_2)\|_{L_x^2}}
\right)
\\
&\leq
C\left[
1+
\frac{t_3-t_1}{t_2-t_1}
\left(
1+\log_+\frac{\|u(t_1)\|_{L_x^2}}{\|u(t_2)\|_{L_x^2}}
\right)
+R(t_3-t_1)
\right],
\label{EQ64}
\end{split}
\end{align}
where we used
\[
\frac{t_3-t_1}{t_2-t_1}R\delta
=
\frac{1}{8}R(t_3-t_1).
\]
Finally,
\[
A\leq C\left(1+R^2(t_3-t_1)^2\right)
\andand
R(t_3-t_1)\leq1+R^2(t_3-t_1)^2.
\]
Combining these bounds with \eqref{EQ63} gives
\begin{align}
\phi(\log4)-\phi(s(t_2))
\leq
C\left(
1+R^2(t_3-t_1)^2
+
\frac{t_3-t_1}{t_2-t_1}
\left(
1+\log_+\frac{\|u(t_1)\|_{L_x^2}}{\|u(t_2)\|_{L_x^2}}
\right)
\right).
\label{EQ65}
\end{align}
Combining \eqref{EQ62} and \eqref{EQ65}, and enlarging $C$ if necessary
to absorb the fixed factor $(2C_*)^{-1}$, we obtain the lower bound in
Theorem~\ref{T01}. The right-hand side of that bound is strictly positive,
and hence~$\|u(t_3)\|_{L_x^2}>0$.
\end{proof}

\begin{proof}[Proof of Theorem~\ref{T02}]
Take $v_c=v$, $v_b=0$, and $R=0$.
The preceding proof applies with
$A=A_n$, and \eqref{EQ64} becomes
\[
a
\leq
C\frac{t_3-t_1}{t_2-t_1}
\left(
1+\log\frac{\|u(t_1)\|_{L_x^2}}{\|u(t_2)\|_{L_x^2}}
\right).
\]
Note that the energy estimate in
Section~\ref{sec0401} shows that
$
\|u(t_2)\|_{L_x^2}\leq\|u(t_1)\|_{L_x^2}
$, which shows that
the logarithm is well-defined.
Together with \eqref{EQ62}--\eqref{EQ63}, this proves~\eqref{EQ03}.
\end{proof}

\section{Sharpness in the bounded-potential case}
\label{sec05}
We now consider the case $p=\infty$ and show that the quadratic dependence
on time in the exponent of Theorem~\ref{T01} is optimal for bounded
complex-valued potentials.

\cole
\begin{Theorem}[A global solution with Gaussian decay]
\label{T03}
There exist a bounded complex-valued potential
\[
  v\in L^\infty((0,\infty)\times\R^3)
\]
such that
$\|v\|_{L^\infty((0,\infty)\times\R^3)}\neq 0$,
constants $B,c>0$, and a nonzero global energy solution
\[
  u\in C([0,\infty);L^2(\R^3))
  \cap L^2_{\text{loc}}(0,\infty;H^1(\R^3))
\]
of
\[
  \partial_tu-\Delta u=vu
  \qquad\text{on }(0,\infty)\times\R^3
\]
such that
\begin{align}\label{EQ204}
  0<\|u(t)\|_{L^2_x}\leq Be^{-ct^2}
  \qquad (t\geq0).
\end{align}
\end{Theorem}
\colb

The construction
for the periodic case,
$\mathbb{T}^{3}$,
follows Meshkov's paper~\cite[Section~3]{Meshkov}.
We include the full argument for completeness, since the subsequent
construction
on $\mathbb{R}^{3}$
uses the precise structure of the functions involved.

\subsection{The switching block}\label{sec0501}

The construction in this section is based on the mechanism introduced by Meshkov in~\cite{Meshkov}. We provide all necessary details, making the presentation self-contained.

\subsubsection{Arithmetic and fixed cutoffs}

For $b\in\N$, set
\begin{align}\label{EQ211}
  \lambda_b=8b+1.
\end{align}
The three-square theorem (see~\cite{HardyWright}, for example) implies that there exists
$n_b\in\mathbb{Z}^3$ with nonnegative components such that
\begin{align}\label{EQ212}
  |n_b|^2=8b+1=\lambda_b.
\end{align}
Indeed, $8b+1$ is not of the form~$4^a(8d+7)$.

We shall repeatedly choose smooth cutoffs on intervals of fixed length.  Every cutoff below is obtained from a fixed smooth cutoff by translation,
and hence all derivatives used below are bounded by an absolute constant.
We shall also use a fixed function
$\phi\in C^\infty(\R)$ such that
\begin{align}\label{EQ213}
  \phi(\theta+2\pi)=\phi(\theta)\comma
  \phi(\theta)=\theta\quad\text{for }|\theta|\leq\frac{\pi}{2},
  \andand
  \|\phi\|_{C^2(\R)}\leq C.
\end{align}
Such a function can be obtained by smoothly joining the two endpoint values on
$[\pi/2,3\pi/2]$ and then extending periodically.

\subsubsection{Construction of one switching block}

The block below switches the frequency from $n_b$ to~$n_{b+1}$.
The construction proceeds in four time steps.  
The first and fourth steps involve a superposition of two modes;
in each case, we separately treat the two single-mode regions,
the two cutoff-transition regions, and the central two-mode region.
The second step removes the spatial phase from a single
mode, and the third changes its temporal decay rate.  Thus every possible
zero occurs in a central two-mode piece, where both modes solve the free heat
equation.

\cole
\begin{Lemma}[Uniform switching block]
\label{L03}
There exist $b_0\in\N$ and constants $C,c_0>0$ with the following property. 
Let $b\geq b_0$, and set
\[
  n=n_b
  \comma
   k=n_{b+1},
  \andand m=n_{b+2},
\]
so that
\begin{align}\label{EQ215}
  |n|^2=\lambda_b
  \comma
   |k|^2=\lambda_b+8,
  \andand |m|^2=\lambda_b+16.
\end{align}
Then there are a smooth function $U_b$, a bounded measurable function $q_b$, and a number $a_b>0$ on $[\lambda_b,\lambda_b+8]\times\T^3$, such that
\begin{align}\label{EQ216}
  (\partial_t-\Delta) U_b=q_bU_b
  \andand
  \|q_b\|_{L^\infty([\lambda_b,\lambda_b+8]\times\T^3)}\leq C,
\end{align}
and
\begin{align}
  U_b(t,x)&=e^{-\lambda_b t+i n\cdot x}
    \comma 
  \lambda_b\leq t\leq\lambda_b+\frac14,
  \label{EQ217}\\
  U_b(t,x)&=a_be^{-(\lambda_b+8)t+i k\cdot x}
  \comma 
  \lambda_b+\frac{31}{4}\leq t\leq\lambda_b+8.
  \label{EQ218}
\end{align}
Moreover, there is a positive continuous piecewise-smooth function $M_b$ on $[\lambda_b,\lambda_b+8]$ such that
\begin{align}\label{EQ219}
  |U_b(t,x)|\leq2M_b(t)
  \andand
  |\nabla_xU_b(t,x)|\leq C\sqrt{\lambda_b}\,M_b(t),
\end{align}
and, away from the finitely many junction points of $M_b$,
\begin{align}\label{EQ220}
  \frac{\text{d}}{\text{d}t}\log M_b(t)\leq-c_0t.
\end{align}
At the two endpoints, $M_b$ satisfies
\[
M_b(\lambda_b)=e^{-\lambda_b^2}
\andand
M_b(\lambda_b+8)=a_b e^{-(\lambda_b+8)^2}
.
\]
\end{Lemma}
\colb

Note that the last assertion
follows directly from \eqref{EQ217} and~\eqref{EQ218}.

In the proof, we denote
\begin{align}\llabel{EQ210}
  d_{2\pi}(s)=\operatorname{dist}(s,2\pi\mathbb{Z}).
\end{align}

\begin{proof}[Proof of Lemma~\ref{L03}]
All constants below are independent of~$b$.
Since the components of $m$
and $n$ are nonnegative, \eqref{EQ215} gives
\begin{align}\label{EQ221}
  \lambda_b+16\leq m\cdot(m+n)\leq C\lambda_b
  \andand
  |m+n|^2\leq C\lambda_b.
\end{align}
For sufficiently large $b$, the quadratic equation
\begin{align}\label{EQ222}
  24-2\alpha m\cdot(m+n)+\alpha^2|m+n|^2=0
\end{align}
has two positive roots, because
\[
  (m\cdot(m+n))^2-24|m+n|^2
  \geq(\lambda_b+16)^2-C\lambda_b>0.
\]
Denote by  $\alpha$ the smaller root.
Rationalizing the numerator gives
\begin{align}\label{EQ223}
  \alpha
  =\frac{24}{\underbrace{m\cdot(m+n)}_{c\lambda_b\leq\cdot\leq C\lambda_b}
  +\sqrt{\underbrace{(m\cdot(m+n))^2-24|m+n|^2}_{0<\cdot\leq C\lambda_b^2}}},
\end{align}
and thus
\begin{align}\label{EQ224}
  \frac{c}{\lambda_b}\leq\alpha\leq\frac{C}{\lambda_b}.
\end{align}
Increasing $b_0$ once and for all, we may also assume
\begin{align}\label{EQ225}
  0<\alpha\leq\frac12
  \andand
  |\alpha\phi(\theta)|\leq\frac\pi4
  \text{~for~}\theta\in\R
  .
\end{align}

\smallskip
\noindent\textbf{Step 1.}
$\lambda_b\leq t\leq\lambda_b+2$.
Define
\begin{align}
  u_1(t,x)&=e^{-\lambda_b t+i n\cdot x},\label{EQ226}\\
  u_2(t,x)&=-d\exp\!\left[-(\lambda_b-8)t-i m\cdot x
  +i\alpha\phi((m+n)\cdot x)\right],\label{EQ227}
\end{align}
where $d=e^{-8(\lambda_b+1)}$ is chosen so that
\begin{align}\label{EQ228}
  |u_1(\lambda_b+1,x)|=|u_2(\lambda_b+1,x)|.
\end{align}
Consequently,
\begin{align}\label{EQ229}
  \frac{|u_2(t,x)|}{|u_1(t,x)|}=e^{8(t-\lambda_b-1)}.
\end{align}
If $d_{2\pi}((m+n)\cdot x)<\pi/2$, the phase in \eqref{EQ227} has
spatial gradient $-m+\alpha(m+n)$.  Equations
\eqref{EQ222} and \eqref{EQ215} give
\[
  |-m+\alpha(m+n)|^2
  =\lambda_b+16-2\alpha m\cdot(m+n)+\alpha^2|m+n|^2
  =\lambda_b-8.
\]
Thus,
\begin{align}\label{EQ230}
  (\partial_t-\Delta) u_2=0
  \quad\text{if }d_{2\pi}((m+n)\cdot x)<\frac\pi2.
\end{align}
On the whole torus, direct differentiation yields
\begin{align}\label{EQ231}
  \frac{(\partial_t-\Delta) u_2}{u_2}
  =&24-2\underbrace{\alpha}_{\leq C\lambda_b^{-1}}\underbrace{\phi'((m+n)\cdot x)}_{\lec 1}
  \underbrace{m\cdot(m+n)}_{\leq C\lambda_b}\notag\\
  &+\underbrace{\alpha^2}_{\leq C\lambda_b^{-2}}\underbrace{\phi'((m+n)\cdot x)^2}_{\lec 1}\underbrace{|m+n|^2}_{\leq C\lambda_b}
  -i\underbrace{\alpha}_{\leq C\lambda_b^{-1}}\underbrace{\phi''((m+n)\cdot x)}_{\lec 1}\underbrace{|m+n|^2}_{\leq C\lambda_b}.
\end{align}
It follows from \eqref{EQ213}, \eqref{EQ221}, and
\eqref{EQ224} that
\begin{align}\label{EQ232}
  |(\partial_t-\Delta) u_2|\leq C|u_2|.
\end{align}
Choose $\psi_1,\psi_2\in C^\infty(\R)$ with values in $[0,1]$ such that
\begin{align*}
  &\psi_1=1\onon{(-\infty,\lambda_b+\frac{3}{2}]},
  \qquad\psi_1=0\onon{[\lambda_b+\frac{7}{4},\infty)},\\
  &\psi_2=0\onon{(-\infty,\lambda_b+\frac{1}{4}]},
  \qquad\psi_2=1\onon{[\lambda_b+\frac{1}{2},\infty)},
\end{align*}
and $|\psi_j'|\leq C$.  Set
\begin{align}\label{EQ233}
  U_b=\psi_1u_1+\psi_2u_2.
\end{align}
We now treat the five time intervals in Step 1.

On
$[\lambda_b,\lambda_b+\frac{1}{4}]$ only $u_1$ is present, and on
$[\lambda_b+\frac{7}{4},\lambda_b+2]$ only $u_2$ is present; hence, the desired quotient bound of $q_b$ follows from $(\partial_t-\Delta) u_1=0$ and~\eqref{EQ232}. 

On the two
cutoff-transition intervals $(\lambda_b+\frac{1}{4},\lambda_b+\frac{1}{2})\cup(\lambda_b+\frac{3}{2},\lambda_b+\frac{7}{4})$, \eqref{EQ229} gives
\begin{align*}
  \frac{|u_2|}{|u_1|}&\leq e^{-4}
  \comma\lambda_b+\frac{1}{4}\leq t\leq\lambda_b+\frac{1}{2},\\
  \frac{|u_1|}{|u_2|}&\leq e^{-4}
  \comma\lambda_b+\frac{3}{2}\leq t\leq\lambda_b+\frac{7}{4}.
\end{align*}
Consequently,
\begin{align*}
  |U_b|&\geq(1-e^{-4})|u_1|
  \andand
  |(\partial_t-\Delta) U_b|\leq C|u_1|\leq C|U_b|
  \comma\lambda_b+\frac{1}{4}\leq t\leq\lambda_b+\frac{1}{2},\\
  |U_b|&\geq(1-e^{-4})|u_2|\andand
  |(\partial_t-\Delta) U_b|\leq C|u_2|\leq C|U_b|
  \comma\lambda_b+\frac{3}{2}\leq t\leq\lambda_b+\frac{7}{4}.
\end{align*}

It remains to check the central piece
$\lambda_b+\frac{1}{2}\leq t\leq\lambda_b+\frac{3}{2}$, where $U_b=u_1+u_2$.  There are two
special cases.  If $d_{2\pi}((m+n)\cdot x)<\pi/2$, then
\eqref{EQ230} gives $(\partial_t-\Delta) U_b=0$, including at every zero of~$U_b$.  If $d_{2\pi}((m+n)\cdot x)\geq\pi/2$, choose the representative
$\theta\in[\pi/2,3\pi/2]$ of $(m+n)\cdot x$.  By
\eqref{EQ225}, we have
  \[\operatorname{dist}(\theta-\underbrace{\alpha\phi(\theta)}
  _{|\cdot|\leq \frac{\pi}{4}},2\pi\mathbb{Z})\geq
\frac{\pi}{4}.\]
Hence,
\begin{align}
\begin{split}
\label{EQ229b}
 \left|\frac{U_b}{u_1}\right|^2
 &=\left|1-e^{8(t-\lambda_b-1)}
 e^{-i(\theta-\alpha\phi(\theta))}\right|^2\\
 &=\left(1-e^{8(t-\lambda_b-1)}\right)^2
 +2e^{8(t-\lambda_b-1)}
 \left(1-\cos(\theta-\alpha\phi(\theta))\right)\\
 &\geq c\left(1+e^{8(t-\lambda_b-1)}\right)^2.
 \end{split}
\end{align}
Combining~\eqref{EQ229} with \eqref{EQ229b}, we obtain 
\begin{align}\label{EQ234}
  |U_b|=\frac{|U_b|}{|u_1|}|u_1|\geq c|u_1|\geq c(|u_1|+|u_2|)
  \onon{[\lambda_b+\frac{1}{2}, \lambda_b+\frac{3}{2}]}.
\end{align}
Together with \eqref{EQ232}, this proves
$|(\partial_t-\Delta) U_b|\leq C|U_b|$ on every part from Step 1.

\smallskip

\noindent\textbf{Step 2.} $\lambda_b+2\leq t\leq\lambda_b+4$.
There is only one mode in this step; the calculation below removes its
spatial correction while keeping the quotient uniformly bounded.  Choose
$\eta\in C^\infty(\R;[0,1])$ such that
\[
  \eta=1\onon{(-\infty,\lambda_b+\frac{5}{2}]},
  \eta=0\onon{[\lambda_b+\frac{7}{2},\infty)},\andandone
  |\eta'|\leq C,
\]
and set
\begin{align}\label{EQ235}
  U_b(t,x)=-d\exp\!\left[-(\lambda_b-8)t-i m\cdot x
  +i\eta(t)\alpha\phi((m+n)\cdot x)\right].
\end{align}
The function never vanishes.  Direct differentiation and
\eqref{EQ221}--\eqref{EQ224} give
\begin{align*}
  \frac{(\partial_t-\Delta) U_b}{U_b}
  =&24-2\eta\alpha\phi'((m+n)\cdot x)m\cdot(m+n)\\
  &+\eta^2\alpha^2\phi'((m+n)\cdot x)^2|m+n|^2
  -i\eta\alpha\phi''((m+n)\cdot x)|m+n|^2\\
  &+i\eta'\alpha\phi((m+n)\cdot x),
\end{align*}
and hence
\begin{align}\label{EQ236}
  \left|\frac{(\partial_t-\Delta) U_b}{U_b}\right|
  \leq 24+C\alpha\lambda_b+C\alpha^2\lambda_b+C\alpha\leq C.
\end{align}
The formulas from Steps 1 and 2 agree on a neighborhood of $t=\lambda_b+2$.

\smallskip
\noindent\textbf{Step 3.} $\lambda_b+4\leq t\leq\lambda_b+6$.
This step contains one nonvanishing mode.  We only change its temporal
decay rate from $\lambda_b-8$ to $\lambda_b+16$.
Put
\begin{align}\label{EQ237}
  d_1=d\exp[-(\lambda_b-8)(\lambda_b+5)].
\end{align}
Choose $\rho\in C^\infty(\R)$ such that
\begin{align}\label{EQ238}
  -7\leq\rho\leq17,
  \quad
  \rho=-7\onon{(-\infty,\lambda_b+\frac{9}{2}]}\comma
  \rho=17\onon{[\lambda_b+\frac{11}{2},\infty)},\andandone
  |\rho'|\leq C.
\end{align}
Define
\begin{align}\label{EQ239}
  U_b(t,x)
  =-d_1\exp\!\left(-(\lambda_b-1+\rho(t))(t-\lambda_b-5)-i m\cdot x\right).
\end{align}
This agrees with \eqref{EQ235} near $t=\lambda_b+4$.  Since
$|m|^2=\lambda_b+16$,
\begin{align}\label{EQ240}
  \frac{(\partial_t-\Delta) U_b}{U_b}
  =17-\rho(t)-\rho'(t)(t-\lambda_b-5),
\end{align}
which is uniformly bounded.  Near $t=\lambda_b+6$, the function has the form
\begin{align}\label{EQ241}
  u_4(t,x)=-d_2e^{-(\lambda_b+16)t-i m\cdot x}
\end{align}
for a constant $d_2=d_1e^{(\lambda_b+16)(\lambda_b+5)}=e^{16(\lambda_b+7)}$.

\smallskip
\noindent\textbf{Step 4.} $\lambda_b+6\leq t\leq\lambda_b+8$.
This step is again divided into five time intervals: a single-mode
interval for $u_4$, a transition introducing $u_5$, the central two-mode
interval, a transition removing $u_4$, and a single-mode interval for $u_5$.
Let
\begin{align}\label{EQ242}
  u_5(t,x)=a_be^{-(\lambda_b+8)t+i k\cdot x},
\end{align}
where $a_b=e^{8(\lambda_b+7)}$ is chosen so that
\begin{align}\label{EQ243}
  |u_4(\lambda_b+7,x)|=|u_5(\lambda_b+7,x)|.
\end{align}
Thus
\begin{align}\label{EQ244}
  \frac{|u_5(t,x)|}{|u_4(t,x)|}=e^{8(t-\lambda_b-7)}.
\end{align}
Both $u_4$ and $u_5$ solve the free heat equation.  Choose
$\zeta_1,\zeta_2\in C^\infty(\R;[0,1])$ such that
\begin{align*}
  &\zeta_1=1\onon{(-\infty,\lambda_b+\frac{15}{2}]}
  \andand\zeta_1=0\onon{[\lambda_b+\frac{31}{4},\infty)},\\
  &\zeta_2=0\onon{(-\infty,\lambda_b+\frac{25}{4}]}
  \andand\zeta_2=1\onon{[\lambda_b+\frac{13}{2},\infty)},
\end{align*}
with $|\zeta_j'|\leq C$, and set
\begin{align}\label{EQ245}
  U_b=\zeta_1u_4+\zeta_2u_5.
\end{align}
On $[\lambda_b+6,\lambda_b+\frac{25}{4}]$ only $u_4$ is present, and on
$[\lambda_b+\frac{31}{4},\lambda_b+8]$ only $u_5$ is present; both solve the free heat
equation.  On the two transition intervals, \eqref{EQ244} gives
\begin{align*}
  \frac{|u_5|}{|u_4|}&\leq e^{-4}
      \comma \lambda_b+\frac{25}{4}\leq t\leq\lambda_b+\frac{13}{2},\\
  \frac{|u_4|}{|u_5|}&\leq e^{-4}
    \comma 
  \lambda_b+\frac{15}{2}\leq t\leq\lambda_b+\frac{31}{4}.
\end{align*}
Thus, exactly as in Step 1, $|(\partial_t-\Delta) U_b|\leq C|U_b|$ where either
cutoff varies.  On $[\lambda_b+13/2,\lambda_b+15/2]$ both cutoffs equal one, so
$(\partial_t-\Delta) U_b=0$, including at all possible zeros.

We have now covered every point of the block.  The only possible zeros lie
in the central free regions identified in Steps 1 and 4.  Define
\begin{align}\label{EQ246}
  q_b(t,x)=
  \begin{cases}
  \dfrac{(\partial_t-\Delta) U_b(t,x)}{U_b(t,x)},&U_b(t,x)\neq0,\\[2mm]
  0,&U_b(t,x)=0.
  \end{cases}
\end{align}
Then $q_b$ is measurable, $|q_b|\leq C$, and
$(\partial_t-\Delta) U_b=q_bU_b$ everywhere.  Since $\zeta_1=0$ and
$\zeta_2=1$ near $t=\lambda_b+8$, formula \eqref{EQ218} follows from~\eqref{EQ245}.

We finish with the envelope estimates.  Define $M_b$ on the four time steps as follows:
in Step 1 take $M_b=|u_1|$ for $t\leq\lambda_b+1$ and
$M_b=|u_2|$ for $t\geq\lambda_b+1$; in Steps 2 and 3 take the modulus of the
single exponential; in Step 4 take $M_b=|u_4|$ for
$t\leq\lambda_b+7$ and $M_b=|u_5|$ for $t\geq\lambda_b+7$. 
The matching conditions above ensure that $M_b$ is continuous. 
Since every cutoff lies in
$[0,1]$,
\[
  |U_b|\leq2M_b.
\]
by
\eqref{EQ221} and \eqref{EQ224}, every spatial phase has gradient bounded by $C\sqrt\lambda_b$; hence
\[
  |\nabla_xU_b|\leq C\sqrt\lambda_b\,M_b.
\]
In Steps 1, 2, and 4, the logarithmic derivative of $M_b$ is one of
$-\lambda_b$, $-(\lambda_b-8)$, $-(\lambda_b+8)$, or $-(\lambda_b+16)$.  In Step 3 it equals
\[
  -(\lambda_b-1+\rho(t))-\rho'(t)(t-\lambda_b-5).
\]
By \eqref{EQ238}, the Step 3 expression is at most $-\lambda_b/2$ when
$b_0$ is sufficiently large.  The four constant rates are also at most
$-\lambda_b/2$.  Since $t\leq\lambda_b+8\leq2\lambda_b$ for such $b$, we obtain
\[
  \frac{\text{d}}{\text{d}t}\log M_b(t)\leq-\frac{t}{4}.
\]
This proves \eqref{EQ219}--\eqref{EQ220}, with $c_0=\frac{1}{4}$.
\end{proof}

\subsection{The global periodic construction}
\label{sec0502}
The passage from a single block to all positive times requires three checks:
Consecutive blocks match smoothly through their free endpoint pieces, the
block envelopes glue to give Gaussian decay, and the resulting bounded
potential is not identically zero.

\cole
\begin{Proposition}[Global periodic mode-switching profile]\label{P03}
There exist a bounded complex-valued function $q$ and a smooth function $W$,
both $2\pi$-periodic in $x$, such that
\begin{align}\label{EQ248}
  q\in L^\infty((0,\infty)\times\R^3),
\end{align}
and
\begin{align}\label{EQ249}
  \partial_tW-\Delta W=qW
  \qquad\text{on }(0,\infty)\times\R^3,
\end{align}
such that, for some $\lambda_0>0$ and $\nu_0\in\mathbb{Z}^3$ with
$|\nu_0|^2=\lambda_0$,
\begin{align}\label{EQ250}
  W(t,x)=e^{-\lambda_0t+i\nu_0\cdot x}
  \comma
  0\leq t\leq\frac14
  ,
\end{align}
and
\begin{align}\label{EQ251}
  |W(t,x)|+|\nabla_xW(t,x)|\leq C_0e^{-c_0t^2}
  \comma
  t\geq0
  \commaone
  x\in\R^3
  .
\end{align}
More precisely, there is a positive, continuous, piecewise smooth function
$M$ on $[0,\infty)$ such that
\begin{align}\label{EQ252}
  |W(t,x)|\leq2M(t)
  \comma
  |\nabla_xW(t,x)|\leq C\sqrt{1+t}\,M(t),
  \andand
  M(t)\leq Ce^{-ct^2}.
\end{align}
Moreover,
\begin{align}\label{EQ253}
  0<\|q\|_{L^\infty((0,\infty)\times\R^3)}<\infty.
\end{align}
\end{Proposition}
\colb

\begin{proof}[Proof of Proposition~\ref{P03}]
Fix $b_0$ as in Lemma \ref{L03}, and extend all functions on $\T^3$
$2\pi$-periodically to~$\R^3$.  We first choose the multiplier for every
block and verify the two junction identities; we then apply the uniform
estimates of Lemma~\ref{L03}.
For $j\geq0$, set
\begin{align*}
  b_j=b_0+j.
\end{align*}
Then,
\begin{align*}
  \lambda_{b_j}=\lambda_{b_0}+8j
  \andand
  \lambda_{b_{j+1}}=\lambda_{b_j}+8.
\end{align*}
The construction in Lemma \ref{L03} gives
$a_b=e^{8(\lambda_b+7)}$.  Define
\begin{align*}
  \gamma_j
  =\exp\left(\lambda_{b_0}^2
  +8j(\lambda_{b_0}+7)+32j(j-1)\right)
    \comma   
  j\geq0
  .
\end{align*}
It follows directly that
\begin{align*}
  \gamma_0=e^{\lambda_{b_0}^2}
  \andand
  \gamma_{j+1}
  =e^{8(\lambda_{b_j}+7)}\gamma_j
  =a_{b_j}\gamma_j.
\end{align*}
On $8j\leq t\leq 8j+8$, define
\begin{align*}
  W(t,x)&=\gamma_jU_{b_j}(t+\lambda_{b_0},x),\\
  M(t)&=\gamma_jM_{b_j}(t+\lambda_{b_0}).
\end{align*}
We first check the junction at $t=8j+8$.  The endpoint formulas
\eqref{EQ217}--\eqref{EQ218} hold on neighborhoods of the endpoints.  Since
$\gamma_{j+1}=a_{b_j}\gamma_j$, the formulas on both sides of the junction
are the same free heat mode
\begin{align*}
  \gamma_{j+1}
  e^{-\lambda_{b_{j+1}}(t+\lambda_{b_0})
  +in_{b_{j+1}}\cdot x}.
\end{align*}
Thus all space-time derivatives of the two pieces agree at $t=8j+8$.
Moreover, the endpoint identities in Lemma \ref{L03} give
\begin{align}
  \begin{split}
  \gamma_jM_{b_j}(\lambda_{b_j}+8)
  &=\gamma_ja_{b_j}e^{-(\lambda_{b_j}+8)^2}
  =\gamma_{j+1}e^{-\lambda_{b_{j+1}}^2}
  =\gamma_{j+1}M_{b_{j+1}}(\lambda_{b_{j+1}}).
  \end{split}
   \llabel{EQ02}
\end{align}
Consequently, $W$ is smooth, while $M$ is positive, continuous, and
piecewise smooth on $[0,\infty)$.
For $8j<t<8j+8$, define
\begin{align*}
  q(t,x)=q_{b_j}(t+\lambda_{b_0},x),
\end{align*}
and set $q=0$ at the junction times. Since the blocks are free heat modes near their endpoints, the equation
continues to hold at every junction. 
In addition,
\begin{align*}
  \|q\|_{L^\infty((0,\infty)\times\R^3)}
  \leq\sup_{j\geq0}
  \|q_{b_j}\|_{L^\infty(
  [\lambda_{b_j},\lambda_{b_j}+8]\times\T^3)}
  \leq C.
\end{align*}
We next verify that the growing multipliers $\gamma_j$ do not affect the
uniform estimates. On the $j$-th block, the same positive constant
$\gamma_j$ multiplies both $U_{b_j}$ and $M_{b_j}$ and is independent of
$t$ and~$x$.  Hence, \eqref{EQ219}--\eqref{EQ220} give
\begin{align*}
  |W(t,x)|
  &\leq2M(t),\\
  |\nabla_xW(t,x)|
  &\leq C\sqrt{\lambda_{b_j}}\,M(t)
   \leq C\sqrt{1+t}\,M(t),\\
  \frac{\text{d}}{\text{d}t}\log M(t)
  &\leq-c_0(t+\lambda_{b_0})
   \leq-c_0t
\end{align*}
away from the corners of~$M$.  Here we used
\begin{align*}
  \lambda_{b_j}=\lambda_{b_0}+8j
  \leq(\lambda_{b_0}+1)(1+t)
    \comma 
  8j\leq t\leq8j+8
  .
\end{align*}
The growth of $\gamma_j$ also produces no growth at a junction.  Indeed,
using $a_{b_j}=e^{8(\lambda_{b_j}+7)}$, we have
\begin{align*}
  \frac{M(8j+8)}{M(8j)}
  &=\frac{\gamma_ja_{b_j}e^{-\lambda_{b_{j+1}}^2}}
  {\gamma_je^{-\lambda_{b_j}^2}}\\
  &=\exp\left(8(\lambda_{b_j}+7)
  -(\lambda_{b_j}+8)^2+\lambda_{b_j}^2\right)\\
  &=e^{-8(\lambda_{b_j}+1)}<1.
\end{align*}
Thus the growth of the matching multiplier is dominated exactly by the
decay already present in the block.
Finally,
\begin{align*}
  M(0)
  =\gamma_0M_{b_0}(\lambda_{b_0})
  =e^{\lambda_{b_0}^2}e^{-\lambda_{b_0}^2}=1.
\end{align*}
Since $\log M$ is continuous and piecewise smooth, the preceding
logarithmic derivative estimate can be integrated across its finitely many
junction points on every compact time interval. Therefore,
\begin{align}\label{EQ254}
  M(t)\leq\exp\!\left(-\frac{c_0}{2}t^2\right).
\end{align}
It follows that
\begin{align*}
  |W(t,x)|+|\nabla_xW(t,x)|
  \leq C(1+\sqrt{1+t})e^{-c_0t^2/2}
  \leq C_0e^{-c_0t^2/4}.
\end{align*}
After decreasing $c_0$ in \eqref{EQ251} if necessary, this proves
\eqref{EQ251}--\eqref{EQ252}.
On the first quarter of the first block, we have
\begin{align*}
  W(t,x)
  &=e^{\lambda_{b_0}^2}
  e^{-\lambda_{b_0}(t+\lambda_{b_0})+in_{b_0}\cdot x}
  =e^{-\lambda_{b_0}t+in_{b_0}\cdot x}.
\end{align*}
Thus \eqref{EQ250} holds with
$\lambda_0=\lambda_{b_0}$ and $\nu_0=n_{b_0}$, for which
$|\nu_0|^2=\lambda_0$.

It remains to show that $q\not\equiv0$. Otherwise, \eqref{EQ249} and the
smoothness of $W$ imply that $W$ is the free heat flow on $\T^3$ with
initial value $e^{in_{b_0}\cdot x}$.  Hence,
\begin{align*}
  \|W(t)\|_{L^2_x(\T^3)}
  =(2\pi)^{3/2}e^{-\lambda_{b_0}t},
\end{align*}
which contradicts \eqref{EQ251} for all sufficiently large~$t$.  This proves~\eqref{EQ253}.
\end{proof}

\subsection{The Euclidean construction}
\label{sec0503}
The periodic profile $W$ is not square integrable on~$\R^3$.
Instead, we look for the
solution of the form $\widetilde u=gW$ with a positive $L^2$ envelope~$g$.  The product rule gives
\begin{align*}
  (\partial_t-\Delta)(gW)
  =gqW+W(\partial_t-\Delta)g-2\nabla g\cdot\nabla W.
\end{align*}
Thus, to obtain a bounded quotient potential, we need
to invent a suitable decaying smooth $g$ such that
\[
|\nabla g\cdot\nabla W|\leq Cg|W|,
\]
including at the zeros of~$W$.
Further away from cancellation, this follows from $|W|\geq cM$ and the ordinary
gradient estimates.  Cancellation can occur only when the two modes in a
switching step have comparable moduli and nearly opposite phases, namely
near $t=8j+1$ and $t=8j+7$. Near these crossings, we isolate the directions $m+n$ and $m+k$ in~$\nabla W$. The envelope will be chosen so that its derivatives in these
directions acquire the same cosine factors that control~$|W|/M$.

In the block $8j\leq t\leq8j+8$, set
\begin{align*}
  b=b_0+j,
  \qquad n=n_b,
  \qquad k=n_{b+1},
  \andandone
  m=n_{b+2}.
\end{align*}
The number $\alpha$ and the modes $u_1,u_2,u_4,u_5$ are those constructed
in Lemma \ref{L03} after the common time shift and multiplication by~$\gamma_j$.  Neither operation changes their modulus ratios.

\cole
\begin{Lemma}[The two crossings in each block]\label{L04}
Let $W$ and $M$ be given by Proposition~\ref{P03}.  In the
block $8j\leq t\leq8j+8$, we have
\begin{align}\label{EQ257}
  c\sqrt{1+t}\leq |m+n|,|m+k|\leq C\sqrt{1+t}.
\end{align}
Define the two possible cancellation regions by
\begin{align}
  \mathcal C_{1,j}
  &=\left\{(t,x):
  |t-(8j+1)|<\frac14\comma
  d_{2\pi}((m+n)\cdot x)<\frac\pi4\right\},
  \label{EQ258}\\
  \mathcal C_{2,j}
  &=\left\{(t,x):
  |t-(8j+7)|<\frac14\comma
  d_{2\pi}((m+k)\cdot x)<\frac\pi4\right\}
  ,
  \label{EQ259}
\end{align}
and set
\begin{align*}
  \mathcal C
  =\bigcup_{j\geq0}
  \left(\mathcal C_{1,j}\cup\mathcal C_{2,j}\right).
\end{align*}
Then
\begin{align}\label{EQ260}
  |W(t,x)|\geq cM(t)
  \qquad\text{on }\mathcal C^c.
\end{align}
On $\mathcal C_{1,j}$,
\begin{align}
  &
  W=u_1+u_2\comma
  \nabla W=inW-i(1-\alpha)(m+n)u_2,
  \label{EQ261}\\
  &
  |u_1|+|u_2|\leq CM(t)\comma
  \bigl(1-\cos((m+n)\cdot x)\bigr)M(t)
  \leq C|W(t,x)|
  ,
  \label{EQ262}
\end{align}
while on $\mathcal C_{2,j}$,
\begin{align}
  &W=u_4+u_5\comma
  \nabla W=-imW+i(m+k)u_5,
  \label{EQ263}\\
  &|u_4|+|u_5|\leq CM(t)\comma
  \bigl(1-\cos((m+k)\cdot x)\bigr)M(t)
  \leq C|W(t,x)|.
  \label{EQ264}
\end{align}
All constants are independent of~$j$.
\end{Lemma}
\colb

\begin{proof}[Proof of Proposition~\ref{L04}]
We prove the frequency estimate, the lower bound away from cancellation,
and the estimates in the two cancellation regions.
Since all components of $m,n,k$ are nonnegative, \eqref{EQ215} gives
\begin{align*}
  \lambda_b+16
  &\leq |m+n|^2
  \leq2(|m|^2+|n|^2)
  \leq4(\lambda_b+16)
  \end{align*}
and
  \begin{align*}
  \lambda_b+16
  &\leq |m+k|^2
  \leq2(|m|^2+|k|^2)
  \leq4(\lambda_b+16).
\end{align*}
Since $\lambda_b=\lambda_{b_0}+8j\sim1+t$ on
$8j\leq t\leq8j+8$, this proves~\eqref{EQ257}.

We next prove~\eqref{EQ260}.  Outside the two central two-mode intervals
\begin{align*}
  \left[8j+\frac12,8j+\frac32\right],
  \qquad
  \left[8j+\frac{13}{2},8j+\frac{15}{2}\right],
\end{align*}
either $|W|=M$, or we are
in a cutoff transition where the mode defining $M$ has coefficient one
and the other mode has modulus at most~$e^{-4}M$.  In the latter case,
\begin{align*}
  |W|\geq M-e^{-4}M=(1-e^{-4})M.
\end{align*}
It therefore remains to consider the two displayed intervals.
On the first one, $W=u_1+u_2$ and
\begin{align*}
  \frac{|u_2|}{|u_1|}=e^{8(t-8j-1)}.
\end{align*}
If $|t-(8j+1)|\geq1/4$, the larger mode is precisely the one defining
$M$, and hence
\begin{align*}
  |W|
  &\geq
  \begin{cases}
    |u_1|-|u_2|\geq(1-e^{-2})M,
    &t\leq8j+\dfrac34,\\[1mm]
    |u_2|-|u_1|\geq(1-e^{-2})M,
    &t\geq8j+\dfrac54.
  \end{cases}
\end{align*}
Suppose now that $|t-(8j+1)|<1/4$ and
$(t,x)\notin\mathcal C_{1,j}$.
Choose $\theta\in[-\pi,\pi]$ representing $(m+n)\cdot x$.  Then
$|\theta|=d_{2\pi}((m+n)\cdot x)\geq\pi/4$, and \eqref{EQ225} gives
\begin{align*}
  d_{2\pi}(\theta-\alpha\phi(\theta))
  &\geq
  \begin{cases}
    (1-\alpha)|\theta|\geq\dfrac\pi8,
    &|\theta|\leq\dfrac\pi2,\\[1mm]
    |\theta|-|\alpha\phi(\theta)|\geq\dfrac\pi4,
    &|\theta|\geq\dfrac\pi2.
  \end{cases}
\end{align*}
Therefore,
\begin{align}\label{EQ265}
\begin{split}
  \frac{|W|^2}{|u_1|^2}
  &=\left|1+\frac{u_2}{u_1}\right|^2=\left|1-e^{8(t-8j-1)}
    e^{-i(\theta-\alpha\phi(\theta))}\right|^2\\
  &=\left(1-e^{8(t-8j-1)}\right)^2
    +2e^{8(t-8j-1)}
    \left(1-\cos(\theta-\alpha\phi(\theta))\right)
  \geq c.
  \end{split}
\end{align}
By the definition of $M$ and the modulus ratio,
\begin{align*}
  e^{-2}M\leq|u_i|\leq M\quad \text{for } i=1,2
  \withwith
  |W|\geq c|u_1|\geq cM.
\end{align*}
On the second central interval, $W=u_4+u_5$ and
\begin{align*}
  \frac{|u_5|}{|u_4|}=e^{8(t-8j-7)}.
\end{align*}
The same dominant-mode argument gives $|W|\geq(1-e^{-2})M$ when
$|t-(8j+7)|\geq1/4$.  If $|t-(8j+7)|<1/4$ and
$(t,x)\notin\mathcal C_{2,j}$, choose $\theta\in[-\pi,\pi]$
representing $(m+k)\cdot x$.  Then
$|\theta|\geq\pi/4$, and
\begin{align*}
  \frac{|W|^2}{|u_4|^2}
  &=\left|1+\frac{u_5}{u_4}\right|^2
  =\left|1-e^{8(t-8j-7)}e^{i\theta}\right|^2\\
  &=\left(1-e^{8(t-8j-7)}\right)^2
    +2e^{8(t-8j-7)}(1-\cos\theta)
  \geq c.
\end{align*}
By the definition of $M$ and the modulus ratio,
\begin{align*}
  e^{-2}M\leq|u_i|\leq M,\quad \text{for } i=4,5,
  \withwith
  |W|\geq c|u_4|\geq cM
  ,
\end{align*}
which proves~\eqref{EQ260}.

It remains to verify the estimates on~$\mathcal C$.
On $\mathcal C_{1,j}$, choose $\theta\in(-\pi/4,\pi/4)$ representing
$(m+n)\cdot x$.  Since $\phi(\theta)=\theta$ and $\phi'(\theta)=1$,
we have
\begin{align}\label{EQ266}
  \frac{u_2}{u_1}
  =-e^{8(t-8j-1)}e^{-i(1-\alpha)\theta}
  ,
\end{align}
and consequently,
\begin{align*}
  \frac{|W|^2}{|u_1|^2}
  &=\left|1+\frac{u_2}{u_1}\right|^2\\
  &=\left(1-e^{8(t-8j-1)}\right)^2
    +2e^{8(t-8j-1)}
    \left(1-\cos((1-\alpha)\theta)\right)\geq c\theta^2.
\end{align*}
Together with $e^{-2}M\leq|u_i|\leq M$, $i=1,2$,  this gives
\begin{align*}
  |W|&\geq c|\theta||u_1|
       \geq c|\theta|M,\\
  |u_1|+|u_2|&\leq2M,\\
  \bigl(1-\cos((m+n)\cdot x)\bigr)M
  &=(1-\cos\theta)M
   \leq C|\theta|M
   \leq C|W|.
\end{align*}
Moreover,
\begin{align*}
  \nabla W
  &=inu_1+i(-m+\alpha(m+n))u_2\\
  &=inW-i(1-\alpha)(m+n)u_2.
\end{align*}
This proves \eqref{EQ261}--\eqref{EQ262}.
On $\mathcal C_{2,j}$, choose $\theta\in(-\pi/4,\pi/4)$ representing
$(m+k)\cdot x$.  In this case,
\begin{align*}
  \frac{|W|^2}{|u_4|^2}
  &=\left|1+\frac{u_5}{u_4}\right|^2
  =\left(1-e^{8(t-8j-7)}\right)^2
    +2e^{8(t-8j-7)}(1-\cos\theta)
   \geq c\theta^2
   \end{align*}
and
  \begin{align*}
  \nabla W
  &=-imu_4+iku_5
   =-imW+i(m+k)u_5.
\end{align*}
Since $e^{-2}M\leq|u_i|\leq M$, for $i=4,5$, the same calculation gives
\begin{align*}
  |W|&\geq c|\theta||u_4|
       \geq c|\theta|M,\\
  |u_4|+|u_5|&\leq2M,\\
  \bigl(1-\cos((m+k)\cdot x)\bigr)M
  &\leq C|W|
  ,
\end{align*}
and \eqref{EQ263}--\eqref{EQ264} is proven.
\end{proof}

For $\nu\in\mathbb{Z}^3\setminus\{0\}$, define the bounded periodic vector field
\begin{align}\label{EQ267}
  A_\nu(x)=\frac{\nu}{|\nu|^2}\sin(\nu\cdot x)
  ,
\end{align}
whose essential feature is the identity
\begin{align}\label{EQ268}
  D_x(x-A_\nu(x))\nu
  =\bigl(1-\cos(\nu\cdot x)\bigr)\nu.
\end{align}
The field in the next lemma is prescribed on a half-unit plateau around
each crossing. Between these plateaus there are only two types of gaps:
the long gap between the two crossings of a block and the short gap
between the second crossing of one block and the first crossing of the next.

\cole
\begin{Lemma}[A global flattening field]\label{L05}
There is a real vector-valued function
\[
  a\in C^\infty([0,\infty)\times\R^3;\R^3)
\]
such that
\begin{align}
  a(t,x)&=A_{m+n}(x)
  \qquad\text{if }|t-(8j+1)|\leq\frac14
  \label{EQ270}
  \end{align}
and
  \begin{align}
  a(t,x)&=A_{m+k}(x)
  \qquad\text{if }|t-(8j+7)|\leq\frac14,
  \label{EQ271}
\end{align}
with
\begin{align}
  |a(t,x)|+|\partial_ta(t,x)|
  &\leq C(1+t)^{-\frac{1}{2}}\label{EQ272}
  \end{align}
and
  \begin{align}
  |D_xa(t,x)|\leq C
  \andand
  |D_x^2a(t,x)|&\leq C(1+t)^{\frac{1}{2}}.
  \label{EQ273}
\end{align}
\end{Lemma}
\colb

\begin{proof}[Proof of Proposition~\ref{L05}]
Fix $h\in C^\infty(\R;[0,1])$ satisfying
\[
  h=0\text{ on }(-\infty,0]
\andand
  h=1\text{ on }[1,\infty).
\]
Set $a=A_{n_{b_0+2}+n_{b_0}}$ on $0\leq t\leq5/4$, and use the
plateau values in \eqref{EQ270}--\eqref{EQ271} at all later crossings.  On the long gap
\[
  8j+\frac54\leq t\leq8j+\frac{27}{4},
\]
define
\begin{align}\label{EQ274}
  a(t,x)={}&\left(1-h\left(\frac{2}{11}
  \left(t-8j-\frac54\right)\right)\right)
  A_{n_{b_0+j+2}+n_{b_0+j}}(x)\notag\\
  &+h\left(\frac{2}{11}
  \left(t-8j-\frac54\right)\right)
  A_{n_{b_0+j+2}+n_{b_0+j+1}}(x).
\end{align}
On the short gap
\[
  8j+\frac{29}{4}\leq t\leq8j+\frac{35}{4},
\]
define
\begin{align}\label{EQ275}
  a(t,x)={}&\left(1-h\left(\frac23
  \left(t-8j-\frac{29}{4}\right)\right)\right)
  A_{n_{b_0+j+2}+n_{b_0+j+1}}(x)\notag\\
  &+h\left(\frac23
  \left(t-8j-\frac{29}{4}\right)\right)
  A_{n_{b_0+j+3}+n_{b_0+j+1}}(x).
\end{align}
The plateaus and the two displayed families of gaps cover $[0,\infty)$.
Since every positive-order derivative of $h$ vanishes at $0$ and $1$,
all pieces agree to every order at the endpoints; hence, $a$ is smooth.
From \eqref{EQ267},
\begin{align*}
  D_xA_\nu
  &=\frac{\nu\otimes\nu}{|\nu|^2}\cos(\nu\cdot x),\\
  D_x^2A_\nu
  &=-\frac{\nu\otimes\nu\otimes\nu}{|\nu|^2}
  \sin(\nu\cdot x),
\end{align*}
and hence
\begin{align*}
  |A_\nu|
  &\leq
  \frac{|\nu|}{|\nu|^2}
  |\sin(\nu\cdot x)|
  \leq|\nu|^{-1},\\
  |D_xA_\nu|
  &\leq
  \frac{|\nu\otimes\nu|}{|\nu|^2}
  |\cos(\nu\cdot x)|
  \leq1,\\
  |D_x^2A_\nu|
  &\leq
  \frac{|\nu\otimes\nu\otimes\nu|}{|\nu|^2}
  |\sin(\nu\cdot x)|
  \leq C|\nu|.
\end{align*}
On either gap, \eqref{EQ274}--\eqref{EQ275} has the form
\begin{align*}
  a=(1-\chi(t))A_\nu+\chi(t)A_\mu,
\end{align*}
where
\begin{align*}
  0\leq\chi\leq1,
  \qquad
  |\chi'|
  \leq
  \max\left\{\frac{2}{11},\frac23\right\}
  \|h'\|_{L^\infty}
  \leq C.
\end{align*}
Each of the frequencies $\nu$ and $\mu$ is the sum of two vectors among
$n_b,n_{b+1},n_{b+2},n_{b+3}$. Hence,
\begin{align*}
  \underbrace{\lambda_b}_{\sim 1+t}
  \leq |\nu|^2,|\mu|^2
  \leq
  \underbrace{4(\lambda_b+24)}_{\sim 1+t}
  \andand
  |\nu|\sim|\mu|\sim\sqrt{1+t}.
\end{align*}
Therefore,
\begin{align*}
  |a|
  &\leq
  |1-\chi|
  \underbrace{|A_\nu|}_{\leq C(1+t)^{-\frac{1}{2}}}
  +|\chi|
  \underbrace{|A_\mu|}_{\leq C(1+t)^{-\frac{1}{2}}}
  \leq C(1+t)^{-\frac{1}{2}},\\
  |\partial_ta|
  &=|\chi'(A_\mu-A_\nu)|\\
  &\leq
  |\chi'|
  \big(
  \underbrace{|A_\mu|}_{\leq C(1+t)^{-\frac{1}{2}}}
  +\underbrace{|A_\nu|}_{\leq C(1+t)^{-\frac{1}{2}}}
  \big)
  \leq C(1+t)^{-\frac{1}{2}},\\
  |D_xa|
  &\leq
  |1-\chi|
  \underbrace{|D_xA_\nu|}_{\leq1}
  +|\chi|
  \underbrace{|D_xA_\mu|}_{\leq1}
  \leq C,\\
  |D_x^2a|
  &\leq
  |1-\chi|
  \underbrace{|D_x^2A_\nu|}_{\leq C|\nu|}
  +|\chi|
  \underbrace{|D_x^2A_\mu|}_{\leq C|\mu|}
  \leq C(1+t)^{\frac{1}{2}}.
\end{align*}
On the plateaus the same bounds hold with $\partial_ta=0$, proving
\eqref{EQ272}--\eqref{EQ273}.
\end{proof}

The proof of the envelope lemma consists of three independent estimates:
comparison with a standard rescaled $L^2$ weight, the usual space-time
derivative bounds, and two normal-derivative estimates on the crossing
plateaus.

\cole
\begin{Lemma}[The slowly expanding envelope]\label{L06}
Let
\[
  G(y)=(1+|y|^2)^{-2},
\]
and, for a sufficiently large fixed $R_0$, define
\begin{align}\label{EQ277}
  g(t,x)=(R_0+t)^{-3/2}
  G\left(\frac{x-a(t,x)}{R_0+t}\right).
\end{align}
Then $g$ is positive and smooth, and
\begin{align}
  &C^{-1}\leq\|g(t)\|_{L^2_x}\leq C,\label{EQ278}\\
  &|\nabla g(t,x)|\leq \frac{C}{R_0+t}g(t,x),\label{EQ279}\\
  &|\partial_tg(t,x)|+|\Delta g(t,x)|\leq Cg(t,x).
  \label{EQ280}
\end{align}
Moreover,
\begin{align}
  |(m+n)\cdot\nabla g(t,x)|
  &\leq Cg(t,x)\frac{|m+n|}{R_0+t}
  \bigl(1-\cos((m+n)\cdot x)\bigr)
  \label{EQ281}
\end{align}
if $|t-(8j+1)|\leq1/4$, while
\begin{align}
  |(m+k)\cdot\nabla g(t,x)|
  &\leq Cg(t,x)\frac{|m+k|}{R_0+t}
  \bigl(1-\cos((m+k)\cdot x)\bigr)
  \label{EQ282}
\end{align}
if $|t-(8j+7)|\leq1/4$.
\end{Lemma}
\colb

\begin{proof}[Proof of Proposition~\ref{L06}]
The function $G$ satisfies
\begin{align}\label{EQ283}
  |\nabla G|+|D^2G|\leq CG
  \andand
  |y\cdot\nabla G(y)|\leq CG(y).
\end{align}
By Lemma \ref{L05}, $|a(t,x)|\leq C$.  Taking $R_0$
sufficiently large gives
\[
  1+\frac{|x-a(t,x)|^2}{(R_0+t)^2}
  \sim
  1+\frac{|x|^2}{(R_0+t)^2}
\]
with uniform constants.  Squaring \eqref{EQ277} and changing
variables $x=(R_0+t)y$ in the two comparison integrals gives
\[
  \|g(t)\|_{L_x^2}^2\sim
  \int_{\R^3}(1+|y|^2)^{-4}\,dy,
\]
which proves~\eqref{EQ278}.
For the following derivations, write
\[
y=\frac{x-a(t,x)}{R_0+t}.
\]
Direct differentiation gives
\begin{align*}
  \nabla g
  &=(R_0+t)^{-5/2}(I-D_xa)^T\nabla G(y),\\
  \Delta g
  &= (R_0+t)^{-7/2}
    \operatorname{tr}
    \left(
        (I-D_xa)^T
        D^2G(y)
        (I-D_xa)
    \right)
    -
    (R_0+t)^{-5/2}
    \nabla G(y)\cdot\Delta a,\\
  \partial_tg
  &=-\frac{3}{2(R_0+t)}g
  +(R_0+t)^{-3/2}\nabla G(y)\cdot
  \left(-\frac{\partial_ta}{R_0+t}-\frac{y}{R_0+t}\right).
\end{align*}
Equations \eqref{EQ283}, \eqref{EQ272}, and
\eqref{EQ273} now give
\[
  |\nabla g|\leq\frac{C}{R_0+t}g
  \andand |\partial_tg|+|\Delta g|\leq Cg.
\]
It remains to prove the two normal-derivative estimates. If
$|t-(8j+1)|\leq1/4$, then \eqref{EQ270} and
\eqref{EQ268} give
\[
  D_x(x-a)(m+n)
  =\bigl(1-\cos((m+n)\cdot x)\bigr)(m+n).
\]
Consequently,
\begin{align*}
  |(m+n)\cdot\nabla g|
  &=(R_0+t)^{-5/2}
  |D_x(x-a)(m+n)\cdot\nabla G(y)|\\
  &\leq Cg\frac{|m+n|}{R_0+t}
  \bigl(1-\cos((m+n)\cdot x)\bigr)
  ,
\end{align*}
which proves~\eqref{EQ281}.  Replacing $m+n$ by $m+k$
and using \eqref{EQ271} then gives~\eqref{EQ282}.
\end{proof}

\begin{proof}[Proof of Theorem \ref{T03}]
Let $W$ be the global switching profile from Proposition
\ref{P03}, and let $g$ be given by Lemma~\ref{L06}.  Set
\begin{align}\label{EQ284}
  \widetilde u(t,x)=g(t,x)W(t,x).
\end{align}
We first prove 
\begin{align}\label{EQ285}
  |\nabla g(t,x)\cdot\nabla W(t,x)|
  \leq Cg(t,x)|W(t,x)|
  ,
\end{align}
which is the only estimate that requires a case-by-case argument.
On $\mathcal C^c$, Lemma \ref{L04},
\eqref{EQ252}, and \eqref{EQ279} imply
\[
  |\nabla g\cdot\nabla W|
  \leq Cg\frac{\sqrt{1+t}}{R_0+t}M(t)
  \leq Cg|W|.
\]
On $\mathcal C_{1,j}$, \eqref{EQ261},
\eqref{EQ279}, and \eqref{EQ281} give
\begin{align*}
  |\nabla g\cdot\nabla W|
  &\leq |n||\nabla g||W|
  +(1-\alpha)|(m+n)\cdot\nabla g|\,|u_2|\\
  &\leq Cg\frac{\sqrt{1+t}}{R_0+t}|W|
  +Cg\frac{|m+n|}{R_0+t}
  \bigl(1-\cos((m+n)\cdot x)\bigr)M(t)\\
  &\leq Cg|W|,
\end{align*}
where the last line uses \eqref{EQ257} and~\eqref{EQ262}.
On $\mathcal C_{2,j}$, \eqref{EQ263},
\eqref{EQ279}, and \eqref{EQ282} give
\begin{align*}
  |\nabla g\cdot\nabla W|
  &\leq |m||\nabla g||W|+|(m+k)\cdot\nabla g|\,|u_5|\\
  &\leq Cg\frac{\sqrt{1+t}}{R_0+t}|W|
  +Cg\frac{|m+k|}{R_0+t}
  \bigl(1-\cos((m+k)\cdot x)\bigr)M(t)\\
  &\leq Cg|W|
\end{align*}
by \eqref{EQ257} and~\eqref{EQ264}.  These three cases prove
\eqref{EQ285}, including at every zero of~$W$.
We now define the potential.  Using \eqref{EQ249},
\eqref{EQ280}, and
\eqref{EQ285}, we obtain the pointwise inequality
\begin{align}\label{EQ286}
  |(\partial_t-\Delta)\widetilde u|
  &=|gqW+W(\partial_tg-\Delta g)
  -2\nabla g\cdot\nabla W|
  \leq C|\widetilde u|.
\end{align}
In particular, the left-hand side vanishes wherever $\widetilde u=0$.
Define
\begin{align}\label{EQ287}
  \widetilde v(t,x)=
  \begin{cases}
  \dfrac{(\partial_t-\Delta)\widetilde u(t,x)}
  {\widetilde u(t,x)},&\widetilde u(t,x)\neq0,\\[2mm]
  0,&\widetilde u(t,x)=0.
  \end{cases}
\end{align}
Then $\widetilde v$ is measurable, $\|{\widetilde v}\|_{L^\infty_{t,x}}\leq C$, and
\begin{align}\label{EQ288}
  \partial_t\widetilde u-\Delta\widetilde u
  =\widetilde v\widetilde u
\end{align}
pointwise and hence in distributions on $(0,\infty)\times\R^3$.
We next verify the global energy class and decay.  Equations
\eqref{EQ252} and \eqref{EQ278} give
\begin{align*}
  \|\widetilde u(t)\|_{L^2_x}
  &\leq2M(t)\|g(t)\|_{L^2_x}
  \leq B_0e^{-ct^2}.
\end{align*}
For each $0\leq t\leq T$,
\begin{align*}
  |\widetilde u(t,x)|&\leq C_Tg(t,x),
  \end{align*}
and
  \begin{align*}
  |\nabla\widetilde u(t,x)|
  &\leq|\nabla g||W|+g|\nabla W|\leq C_Tg(t,x).
\end{align*}
Together with \eqref{EQ278}, this proves
$\widetilde u\in L^2(0,T;H^1)$.  Moreover, the explicit formula
\eqref{EQ277} and the boundedness of $a$ give
\[
  |\widetilde u(t,x)|\leq C_T(1+|x|^2)^{-2}
    \comma  0\leq t\leq T.
\]
The right-hand side is in $L^2(\R^3)$, while $g$ and $W$ are pointwise
continuous in~$t$.  Dominated convergence therefore gives
\[
  \widetilde u\in C([0,\infty);L^2)
  \cap L^2_{\text{loc}}(0,\infty;H^1).
\]
By \eqref{EQ250} and $g>0$, we have $\widetilde u(0)\neq0$.
By continuity, there exists $t_2>0$ such that
$\|\widetilde u(t)\|_{L^2_x}>0$ for $0\leq t\leq t_2$.
For every $t>t_2$, Theorem~\ref{T01}, applied with
$p=\infty$, $t_1=0$, the above $t_2$, and $t_3=t$, gives
$\|\widetilde u(t)\|_{L^2_x}>0$. Therefore,
\begin{align}\label{EQ289}
  0<\|\widetilde u(t)\|_{L^2_x}
  \leq B_0e^{-ct^2}
    \comma t\geq0
\end{align}
Finally, $\widetilde v\not\equiv0$. Indeed, if
$\widetilde v=0$ almost everywhere, Theorem~\ref{T01} with $M=0$
would give an exponential lower bound in $t$ for
$\|\widetilde u(t)\|_{L^2_x}$, contradicting \eqref{EQ289} for large~$t$.
Taking $u=\widetilde u$, $v=\widetilde v$, and $B=B_0$ proves the theorem.
\end{proof}

\appendix
\section{Appendix}
\cole
\begin{Lemma}[Coercivity and Green identity on the compact graph domain]
\label{L02}
Let
\[
h(s)=as+\frac A2s^2,
\wherewhere A>0,
\]
and set
\[
L_h=\partial_s+\mathcal{H}-h'(s).
\]
\begin{enumerate}
\item If $w\in L^2(\R^{1+n})$, $\supp_s w$ is compact, and $L_hw\in
L^2$ in the sense of distributions, then
\begin{align}
\|w\|_2\leq A^{-\frac{1}{2}}\|L_hw\|_2.
\label{EQ21}
\end{align}

\item If, in addition, $z\in L^2$, $\supp_s z$ is compact, and $L_h^*z\in L^2$, then
\begin{align}
\langle L_hw,z\rangle_{L^2}=\langle w,L_h^*z\rangle_{L^2}.
\label{EQ22}
\end{align}
\end{enumerate}
\end{Lemma}
\colb

\begin{proof}[Proof of Proposition~\ref{L02}]
Let
\[
P_N=\mathbf{1}_{[0,N]}(\mathcal{H}).
\]
Since $\mathcal{H}$ has discrete spectrum, $P_N$ has finite-dimensional range.  It commutes with $\mathcal{H}$, $\partial_s$, and multiplication by~$h'(s)$.  Set $w_N=P_Nw$.  Then
\[
L_hw_N=P_NL_hw\in L^2.
\]
On the compact time support of $w$, the function $h'$ is bounded, while $\mathcal{H}$ is bounded on the range of~$P_N$.  Consequently,
\[
\partial_s w_N=L_hw_N-(\mathcal{H}-h')w_N\in L^2,
\]
so $w_N\in H^1_sL^2_y$ and the following Hilbert-space integration by parts is legitimate:
\begin{align*}
\|L_hw_N\|_2^2
={}&\|\partial_sw_N\|_2^2+\|(\mathcal{H}-h')w_N\|_2^2
 +2\operatorname{Re}\int_{\R}
 \langle \partial_sw_N,(\mathcal{H}-h')w_N\rangle_{L^2_y}\,ds\\
={}&\|\partial_sw_N\|_2^2+\|(\mathcal{H}-h')w_N\|_2^2
 +\int_{\R^{1+n}}h''(s)|w_N|^2\,dy\,ds.
\end{align*}
The boundary term vanishes because $w_N$ has compact time support.  Since $h''=A$,
\[
A\|w_N\|_2^2\leq\|P_NL_hw\|_2^2.
\]
Letting $N\to\infty$ gives~\eqref{EQ21}.
For the Green identity, put $z_N=P_Nz$.  Both $w_N$ and $z_N$ belong to $H^1_sL^2_y$, have compact time support, and take values in a finite Hermite spectral subspace.  Hence,
\[
\langle L_hw_N,z_N\rangle=\langle w_N,L_h^*z_N\rangle.
\]
Using
\[
L_hw_N=P_NL_hw
\andand
L_h^*z_N=P_NL_h^*z,
\]
and the strong convergence $P_N\to\operatorname{Id}$ on $L^2$, we obtain~\eqref{EQ22}.
\end{proof}

\colb
\section*{Acknowledgments} 
IK and QX were supported in part by the NSF grant DMS-2205493
and the Simons grant SFI-MPS-TSM-00014233.
The authors used ChatGPT for language editing. All mathematical content was written and independently verified by the authors.

\small
\medskip\medskip
\noindent
I.~Kukavica\\
{Department of Mathematics, University of Southern California, Los Angeles, CA 90089}\\
e-mail: kukavica@usc.edu

\medskip\medskip
\noindent
Q.~Xu\\
{Department of Mathematics, University of Southern California, Los Angeles, CA 90089}\\
e-mail: xuqi@usc.edu

\end{document}